\documentclass[a4paper]{article}
\usepackage[utf8]{inputenc}
\usepackage{amsmath,amssymb,amsthm}
\usepackage{color}
\usepackage{graphicx}
\usepackage{subfigure}
\usepackage{longtable}
\usepackage{multirow}
\usepackage{ulem} 

\usepackage{tikz}
\usetikzlibrary{intersections, calc, arrows.meta}

\newtheorem{Proposition}{Proposition}[section]
\newtheorem{Theorem}[Proposition]{Theorem}
\newtheorem{Lemma}[Proposition]{Lemma}

\newtheorem{Claim}[Proposition]{Claim}
\theoremstyle{definition}
\newtheorem{Remark}[Proposition]{Remark}

\newtheorem{Example}[Proposition]{Example}
\newtheorem{Problem}[Proposition]{Problem}

\title{Prehomogeneous vector spaces obtained from triangle arrangements}
\author{Takeyoshi Kogiso\thanks{Department of Mathematics, Josai University, 1-1 Keyakidai, Sakado, Saitama, 350-0295,
Japan. e-mail: kogiso@math.josai.ac.jp}
\ and Hideto Nakashima\thanks{The Institute of Statistical Mathematics
Midori-cho 10-3, Tachikawa, Tokyo 190-8562, 
Japan. e-mail: hideto@ism.ac.jp}}
\date{}

\newcommand{\C}{\mathbb{C}}
\newcommand{\pmat}[1]{\begin{pmatrix}#1\end{pmatrix}}
\newcommand{\smat}[1]{\left(\begin{smallmatrix}#1\end{smallmatrix}\right)}

\newcommand{\bs}[1]{\boldsymbol{#1}}
\newcommand{\ds}{\displaystyle}

\newcommand{\set}[2]{\left\{#1;\,#2\right\}}
\newcommand{\innV}[2]{\left\langle#1\,\middle|\,#2\right\rangle}

\begin{document}

\maketitle

\begin{abstract}
In this paper, we construct a new series of prehomogeneous vector spaces
from figures made up of triangles, called triangle arrangements.
Our main theorem states that,
under suitable assumptions,
we are able to construct a prehomogeneous vector space obtained from
a triangle arrangement 
by attaching two triangle arrangements corresponding to prehomogeneous vector spaces at a vertex.
We also give examples of 
prehomogeneous vector spaces obtained from triangle arrangements.
Many of them seem to be new.
\end{abstract}

\section*{Introduction}

The theory of prehomogeneous vector spaces, constructed by M.\ Sato \cite{MSato} 
(see also Sato--Shintani \cite{SatoShintani}, Kimura \cite[Introduction]{Kimura})
enables us to construct zeta functions satisfying functional equations systematically.
The key fact is that
basic relative invariants satisfy a local functional equation,
that is,
the Fourier transform of a product of complex powers of basic relative invariants
is essentially given by a product of complex powers of some polynomials.
It is known that 
some polynomials which are not basic relative invariants of any prehomogeneous vector space satisfy a functional equation
(cf.\ Faraut--Kor{\'a}nyi~\cite{FK94}, Kogiso--Sato~\cite{KogisoSato1,KogisoSato2}).
Local functional equations are also studied in the fields of algebraic geometry and projective geometry (cf.\ Etingof--Kazhdan--Polishchuk~\cite{EKP}),
and in these fields, 
they are related homaloidal polynomials which are homogeneous polynomials whose gradient-log maps are bi-rational.
Many authors including \cite{ChaputSabatino,CRS19,CRS08,D,IK,MS} deal with homaloidal polynomials,
and basic relative invariants of regular prehomogeneous vector spaces are recognized as
good examples of homaloidal polynomials (cf.\ \cite{EKP}).
Therefore,
finding new concrete examples of regular prehomogeneous vector spaces
is important both for the theory of prehomogeneous vector spaces and algebraic geometry.

In this paper,
we construct a new series of prehomogeneous vector spaces
from figures made up of triangles, called \textit{triangle arrangements}.
Our main theorem, Theorem~\ref{theo} states that,
under suitable assumptions,
we are able to construct a prehomogeneous vector space obtained from
a triangle arrangement 
by attaching two triangle arrangements corresponding to prehomogeneous vector spaces at a vertex.
We also give examples of 
prehomogeneous vector spaces obtained from triangle arrangements
in Section~\ref{sect:example}.
Combining results in Section~\ref{sect:example} with Theorem~\ref{theo},
we are able to construct a lot of prehomogeneous vector spaces.
Many of them seem to be new.

We organize this paper as follows.
Section~\ref{sect:preliminaries} collects a basic tool that we need later.
In particular,
a notion of triangle arrangements is introduced. 
In Section~\ref{sect:triangulation},
we view triangulation of convex polygons as triangle arrangements and consider which triangulation corresponds to a prehomogeneous vector space.
Section~\ref{sect:structure} is devoted to
study a structure of Lie algebras corresponding to triangle arrangements which have no edge sharing.
Our main theorem, Theorem~\ref{theo} is stated and proved in Section~\ref{sect:main}.
In Section~\ref{sect:example},
we give four examples of triangle arrangements which correspond to prehomogeneous vector spaces.






\subsubsection*{Acknowledgments.}
The fist author is supported by 
the Grant-in-Aid of scientific research of JSPS No.\ 21K03169.
The second author was supported by the Grant-in-Aid for JSPS fellows (2018J00379).
\section{Preliminaries}
\label{sect:preliminaries}

Let $V=\C^n$.
We denote the natural representation of $GL(V)=GL(n,\C)$ on $V$ by $\rho$.
For a given homogeneous polynomial $p(x)$ on $V$,
we introduce a group $G[p]:=GL(1)\times G_0[p]$ where
\[
G_0[p]:=\set{g\in GL(V)}{p\bigl(\rho(g)x\bigr)=p(x)\ \text{for all }x\in V}.
\]
Then,
it is easily verified that $G[p]$ is an algebraic subgroup of $GL(V)$.
By definition,
we see that $p(x)$ is relatively invariant under the action of $G[p]$.
We also use the symbol $\rho$ for the action of $G[p]$ on $V$.
In this paper,
we work on the following problem.

\begin{Problem}
For which homogeneous polynomial $p(x)$
a triplet $(G[p],\rho,V)$ admits a structure of a prehomogeneous vector space?
\end{Problem}

As in \cite{Kimura},
the prehomogeneity is an infinitesimal condition so that
we shall describe the condition of a triplet $(G[p],\rho,V)$ being a prehomogeneous vector space
in terms of Lie algebra.
Let $\mathfrak{g}_0[p]$ be the Lie algebra corresponding to $G_0[p]$.
A bilinear form $\innV{\cdot}{\cdot}$ on $V$ is defined to be
\[\innV{x}{y}={}^{t\!}xy=\sum_{i=1}^nx_iy_i\quad(x,y\in V).\]
Then, we have
\[
\mathfrak{g}_0[p]:=\set{M\in\mathfrak{gl}(V)}{\innV{d\rho(M)x}{\nabla_x p(x)}=0\text{ for all }x\in V},
\]
where $d\rho$ is a differential of $\rho$.
Thus, the Lie algebra $\mathfrak{g}[p]$ of $G[p]$ is given as
\begin{equation}
\label{eq:InvLie}
\mathfrak{g}[p]=\mathfrak{gl}(1)\,\dot+\,\mathfrak{g}_0[p].
\end{equation}
Here,
the symbol $\dot+$ means a direct sum of vector spaces.
By \cite[Proposition 2.2]{Kimura},
we see that the condition of the triplet $(G[p],\rho,V)$ being a prehomogeneous vector space
is described by using its Lie algebra $\mathfrak{g}[p]$ as follows.

\begin{Lemma}[{cf.\ \cite[Proposition 2.2]{Kimura}}]
\label{lemma}
The triplet $(\mathfrak{g}[p],d\rho,V)$ admits a structure of prehomogeneous vector space
if and only if
linear maps $A(x)\colon\mathfrak{g}[p]\to V$ $(x\in V)$, defined by $A(x)M:=d\rho(M)x$ $(M\in\mathfrak{g}[p])$
have full generic rank.
\end{Lemma}

In what follows,
we concentrate the case of homogeneous polynomials $p(x)$ of degree three,
in particular,
those constructed from figures made up of triangles.

Triangle arrangements are figures made up of triangles in such a way that
finite triangles are glued at some vertices or some edges.
In what follows, the symbol $\mathtt{T}$ denote triangle arrangements.
We label number $1,2,3,\dots$ to  vertices of a triangle arrangement $\mathtt{T}$.
We assign variable $x_i$ to vertex $i$ and,
to each triangle with vertices $i,j,k$ in $\mathtt{T}$, we associate a monomial $x_ix_jx_k$.
Then, we construct a polynomial $p(x)$ from $\mathtt{T}$ by
summing up monomials $x_ix_jx_k$ with respect to each triangle with vertices $i,j,k$ in $\mathtt{T}$.
We call $p(x)$ a polynomial with respect to $\mathtt{T}$, or more simply a polynomial of $\mathtt{T}$.

For brevity,
we call a polynomial $p(x)$ is prehomogeneous if
the triplet $(\mathfrak{g}[p],d\rho,V)$ is a prehomogeneous vector space.
Moreover,
if $p(x)$ is obtained from a triangle arrangement, then we also say that $\mathtt{T}$ is prehomogeneous.
In this case,
we often write $\mathfrak{g}[\mathtt{T}]$ instead of $\mathfrak{g}[p]$,
where $p(x)$ is a polynomial of $\mathtt{T}$.

\begin{Example}
\label{exam:TrigArr}
The following figures are three examples of triangle arrangements.
\begin{center}
\begin{tabular}{c@{\hspace{3em}}c@{\hspace{3em}}c}
\includegraphics[scale=0.5]{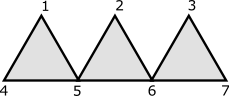}
&
\includegraphics[scale=0.5]{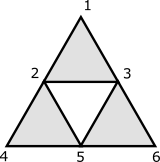}
&
\includegraphics[scale=0.5]{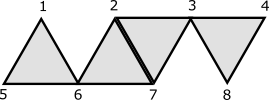}\\
$\mathtt{T}_A$&$\mathtt{T}_B$&$\mathtt{T}_C$
\end{tabular}
\end{center}
If we assign a monomial to each grayed triangles,
then the corresponding polynomials are given as follows.
\[
\begin{array}{l}
p_A(x)=x_1x_4x_5+x_2x_5x_6+x_3x_6x_7\\
p_B(x)=x_1x_2x_3+x_2x_4x_5+x_3x_5x_6\\
p_C(x)=x_1x_5x_6+x_2x_6x_7+x_2x_3x_7+x_3x_4x_8
\end{array}
\]
\end{Example}

Let $\mathtt{T}$ be a triangle arrangement with $n$ vertices.
Set
\[
\mathcal{T}:=\set{T=\{i,j,k\}\subset [n]}{\text{a triangle of vertices $i,j,k$ is contained in $\mathtt{T}$}},
\]
where $[n]:=\{1,2,\dots,n\}$.
We call $\mathcal{T}$ a hypergraph with respect to $\mathtt{T}$.
For each vertex $i\in [n]$,
the set $\mathcal{T}(i)$ consists of $T\in\mathcal{T}$ including $i$, that is,
\[
\mathcal{T}(i):=\set{T\in\mathcal{T}}{i\in T}.
\]
If a vertex $i$ satisfies $\sharp\mathcal{T}(i)=1$, 
then $i$ is said to be an \textit{isolated vertex}.
If $\mathcal{T}$ contains two triangles $T_1$, $T_2$  such that
$\sharp(T_1\cap T_2)=2$, then we say that $\mathtt{T}$ has edge sharing.

For example,
$\mathtt{T}_A$ in Example \ref{exam:TrigArr} have
\[
\mathcal{T}=\bigl\{\{1,4,5\},\,\{2,5,6\},\,\{3,6,7\}\bigr\},\quad
\mathcal{T}(5)=\bigl\{\{1,4,5\},\,\{2,5,6\}\bigr\},
\]
and isolated vertices are $\{1,2,3,4,7\}$.
The triangle arrangement $\mathtt{T}_A$ does not have an edge sharing,
whereas $\mathtt{T}_C$ does.

\section{Triangulation of convex polygons}
\label{sect:triangulation}

In this section, we view triangulation of convex polygons as triangle arrangements
and consider those prehomogeneity.
Since prehomogeneity is independent of the action of $GL(V)$ on $p(x)$,
we first make a reduction of triangulation of polygons 
in order to decrease cases which we consider.
Let us explain this reduction by a concrete example.

The polynomial $p(x)$ associated with Figure \ref{fig:explanation} (left) is described as
\[
p(x)=x_1x_2x_3+x_1x_3x_4+x_1x_4x_5+x_1x_5x_6.
\]
If we change variables 
\[
z_2=x_2+x_4,\quad z_i=x_i\quad(i=1,3,4,5,6),
\]
then $p(x)$ transfers to
\[
p(x)=x_1x_3(x_2+x_4)+x_1x_4x_5+x_1x_5x_6=z_1z_2z_3+z_1z_4z_5+z_1z_5z_6.
\]
Thus, we can decrease numbers of monomials.
In terms of figures,
we focus on the vertex $2$ and a triangle $123$, 
and vanish a triangle sharing edges with the triangle $123$.
This polynomial can be further transferred by changing variables
\[
w_6=z_6+z_4,\quad
w_i=z_i\quad(i=1,2,3,4,5)
\]
to
\[
p(z)=z_1z_2z_3+z_1z_5(z_4+z_6)=w_1w_2w_3+w_1w_5w_6,
\]
and we finally get a polynomial consisting of two monomials.
Along this reduction,
triangles having the vertex $4$ disappear.
Since the original triangulation have $6$ vertex,
we should calculate for $5$-variable polynomial $p(w)$ on $6$-dimensional vector space $\C^6$.

\begin{figure}[b]
\centering
    \begin{tabular}{ccccc}
    \includegraphics[scale=0.17]{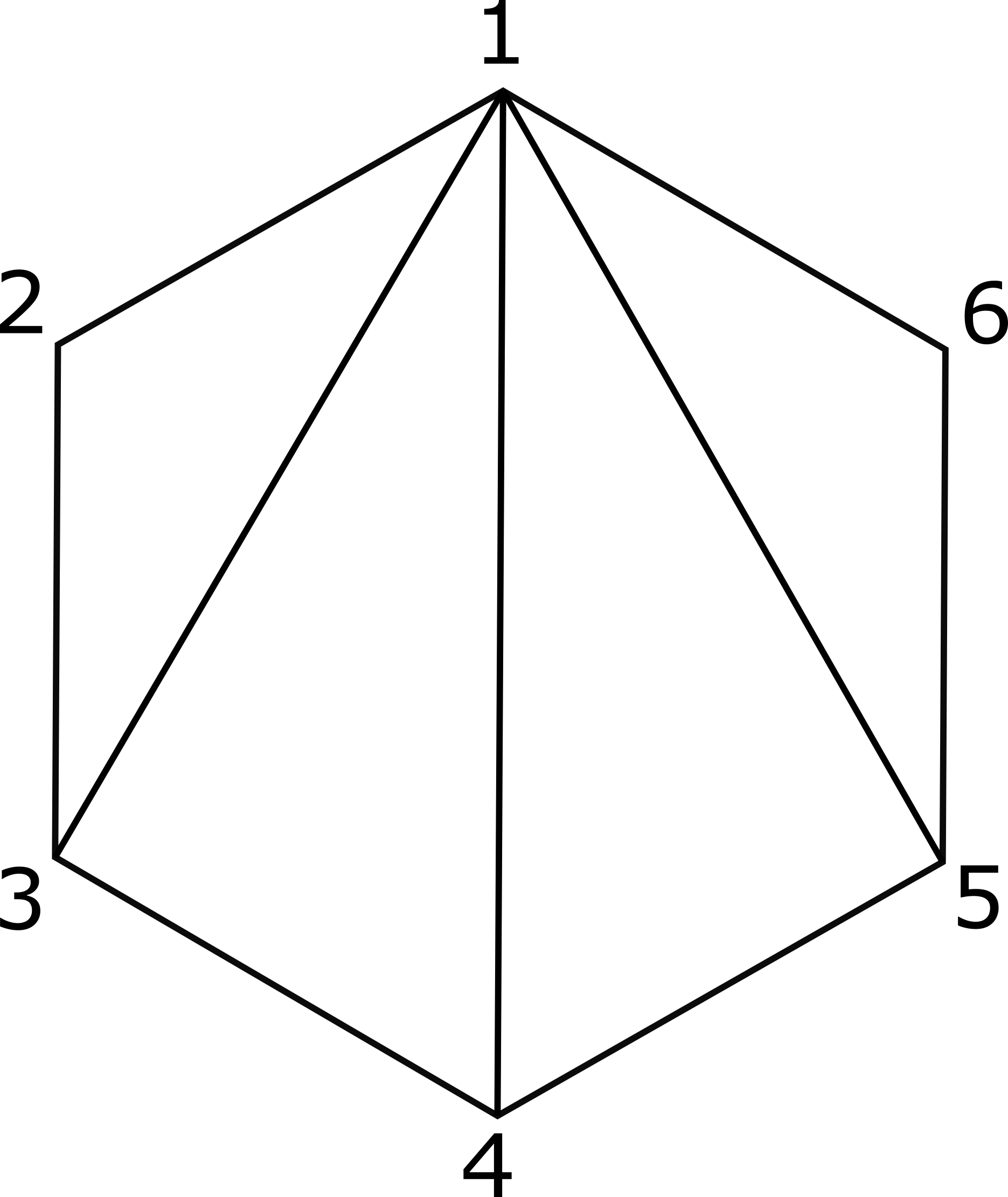}
    &&
    \includegraphics[scale=0.17]{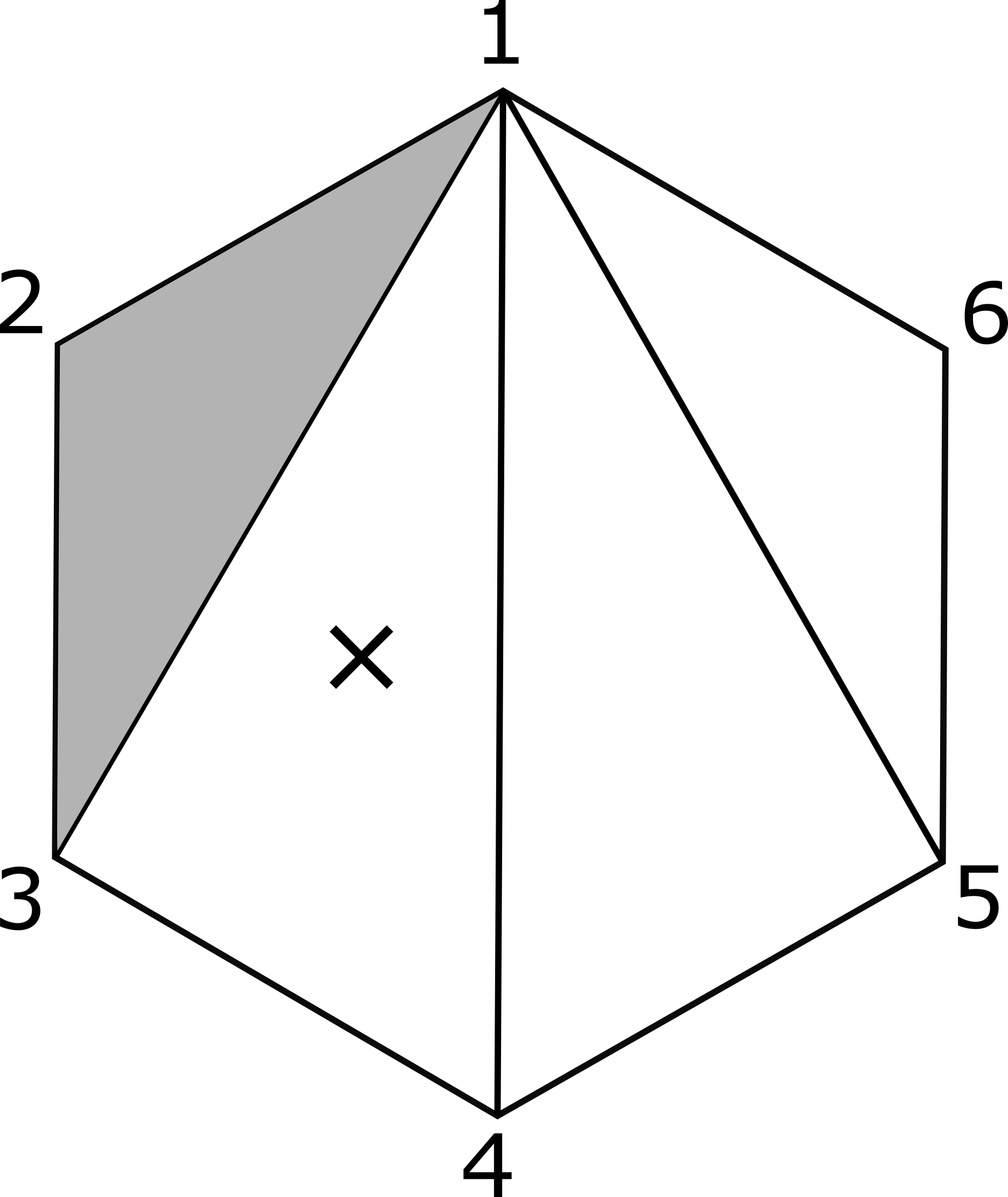}
    &&
    \includegraphics[scale=0.17]{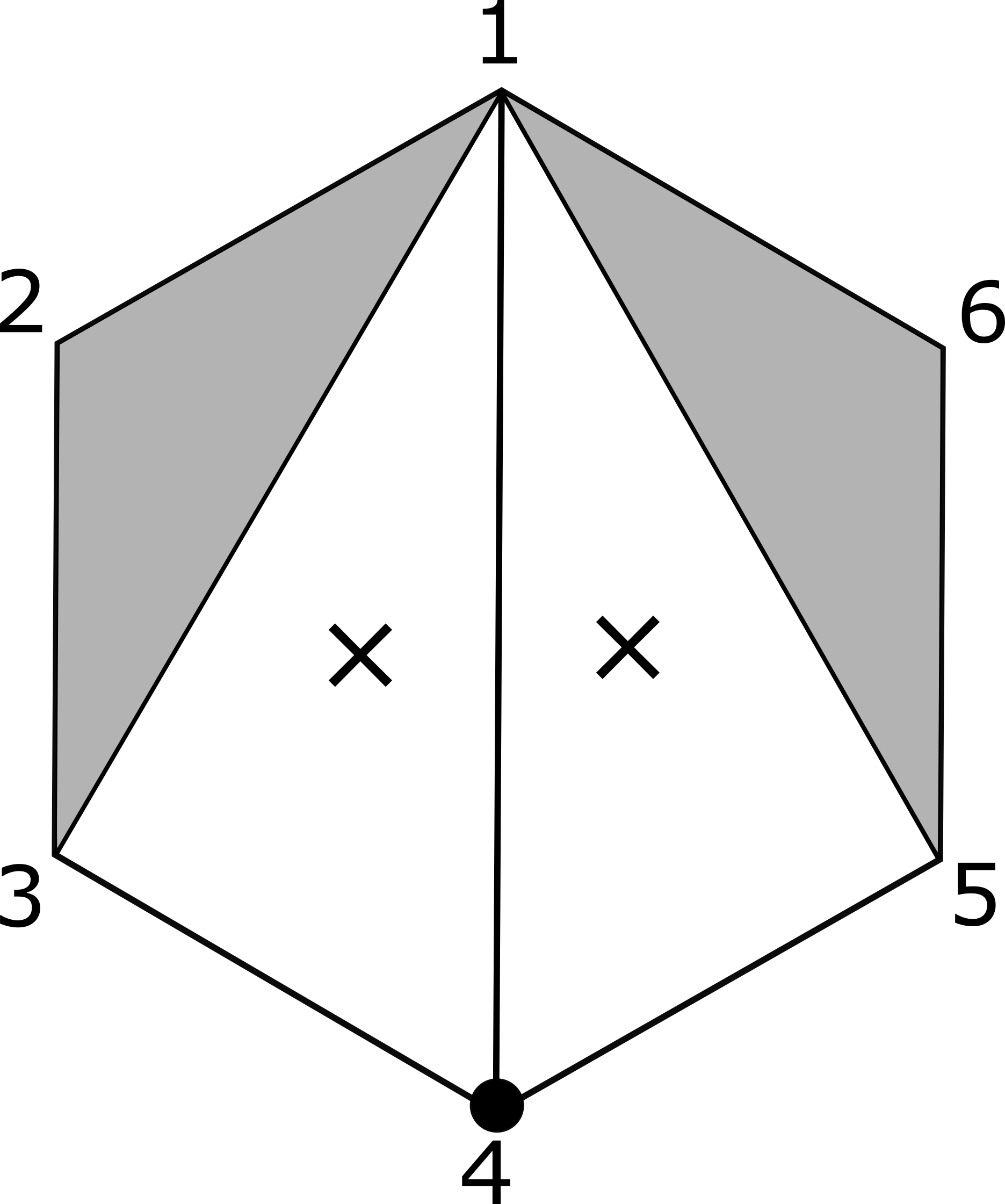}
    \end{tabular}
    \caption{Reduction of a hexagon triangulation}
    \label{fig:explanation}
\end{figure}

In Figure~\ref{fig:trig}, we exhibit triangle arrangements obtained by reduction of triangulation of $n$-polygons up to $n=10$.
We can see prehomogeneity at the top of each figure.
Its proofs are left to Section~\ref{sect:example}.
A black circle $\bullet$ in figures indicates a vertex which does not appear as a vertex of triangles, like the vertex $4$ in the above example.
For simplicity, 
triangle arrangements with a black circle are also called just 
\textit{triangle arrangements}.
Note that, if a triangle arrangement $\mathtt{T}$ is a triangle arrangement $\mathtt{T}'$ with a black circle ($\mathtt{T}'$ does not have a black circle),
then the corresponding Lie algebras $\mathfrak{g}[\mathtt{T}]$ and $\mathfrak{g}[\mathtt{T}']$ are related as
\[
\mathfrak{g}[\mathtt{T}]=\set{M=\pmat{M'&0\\ {}^t\bs{x}&m}}{M'\in\mathfrak{g}[\mathtt{T}'],\,\bs{x}\in\C^{\sharp \mathtt{T}'},\,m\in\C}.
\]
In particular,
the prehomogeneity of $\mathtt{T}$ is the same as that of $\mathtt{T}'$.

We see from Figure~\ref{fig:trig} that
triangle arrangements obtained by reduction of triangulation of $n$-gons may unconnected.
We shall show in Proposition~\ref{prop:unconnected} 
that they cannot be prehomogeneous unless one of the connected components is a black circle discussed in the previous paragraph.

Table~\ref{table:numbers} includes numbers of 
(A) triangulation of $n$-polygons under rotations and reflections, 
(B) those of reduced triangle arrangements and 
(C) those of prehomogeneous vector spaces.

\begin{table}
    \centering
    \begin{tabular}{l|*{15}{c}}
    $n$&6&7&8&9&10&11&12&13&14&15&16&17\\ \hline
    \text{(A)}&3&4&12&27&82&228&733&2282&7528&24834&83898&285357\\
    \text{(B)}&3&2&7&7&26&37&137&298&993&2726&8749&26446\\
    \text{(C)}&2&2&4&3&9&7&23&18&61&56&174&186
    \end{tabular}
    \caption{The row of (A) indicates numbers of triangulation of $n$-polygons up to rotations and reflections, 
    (B) those of reduced triangle arrangements and 
    (C) those of prehomogeneous vector spaces.}
    \label{table:numbers}
\end{table}

\section{Structure of $\mathfrak{g}[p]$ without edge sharing}
\label{sect:structure}

Let $\mathtt{T}$ be a triangle arrangement with $n$ vertices and $p(x)$ the corresponding polynomial.
Suppose that $\mathtt{T}$ has no edge sharing.
In this section,
we investigate a structure of $\mathfrak{g}[p]$,
which will be needed to prove our main theorem.
By \eqref{eq:InvLie}, it is enough to calculate $\innV{d\rho(M)x}{\nabla_xp(x)}=0$.
By definition,
$p(x)$ can be described by using the hyper graph $\mathcal{T}$ associated with $\mathtt{T}$ as
\[
p(x)=\sum_{\{i,j,k\}\in\mathcal{T}}x_ix_jx_k.
\]
Thus, we have
\begin{equation}
\label{eq:expand}
\begin{array}{r@{\ }c@{\ }l}
\innV{d\rho(M)x}{\nabla_xp(x)}
&=&
\ds
\sum_{i=1}^n\sum_{a=1}^nM_{ia}x_a\cdot \sum_{\{i,j,k\}\in\mathcal{T}(i)}x_jx_k\\
&=&
\ds
\sum_{i=1}^n\sum_{a=1}^n\sum_{\{i,j,k\}\in\mathcal{T}(i)}M_{ia}x_ax_jx_k.
\end{array}
\end{equation}

We first exhibit a calculation of a Lie algebra of a concrete polynomial.

\begin{Example}
\label{exam:lie}
Let $p(x)$ be a homogeneous polynomial whose hyper graph is given as $\mathcal{T}=\{\{1,2,3\},\{1,4,5\}\}$,
that is,
\[
p(x)=x_1x_2x_3+x_1x_4x_5.
\]
In this case, we have
\[
\begin{array}{l}
\innV{d\rho(M)x}{\nabla_xp(x)}\\
\quad=M_{11}x_1(x_2x_3+x_4x_5)
+M_{12}x_2(x_2x_3+x_4x_5)+M_{13}x_3(x_2x_3+x_4x_5)\\
\quad\qquad
+M_{14}x_4(x_2x_3+x_4x_5)
+M_{15}x_5(x_2x_3+x_4x_5)\\
\qquad
+M_{21}x_1\cdot x_1x_3
+M_{22}x_2\cdot x_1x_3
+M_{23}x_3\cdot x_1x_3
+M_{24}x_4\cdot x_1x_3
+M_{25}x_5\cdot x_1x_3\\ \qquad
+M_{31}x_1\cdot x_1x_2
+M_{32}x_2\cdot x_1x_2
+M_{33}x_3\cdot x_1x_2
+M_{34}x_4\cdot x_1x_2
+M_{35}x_5\cdot x_1x_2\\ \qquad
+M_{41}x_1\cdot x_1x_5
+M_{42}x_2\cdot x_1x_5
+M_{43}x_3\cdot x_1x_5
+M_{44}x_4\cdot x_1x_5
+M_{45}x_5\cdot x_1x_5\\ \qquad
+M_{51}x_1\cdot x_1x_4
+M_{52}x_2\cdot x_1x_4
+M_{53}x_3\cdot x_1x_4
+M_{54}x_4\cdot x_1x_4
+M_{55}x_5\cdot x_1x_4
\end{array}\]
and hence
\[
\begin{array}{l}
\innV{d\rho(M)x}{\nabla_xp(x)}\\
\quad=
(M_{11}+M_{22}+M_{33})x_1x_2x_3
+
(M_{11}+M_{44}+M_{55})x_1x_4x_5\\
\qquad
+(M_{24}+M_{53})x_1x_3x_4+(M_{25}+M_{43})x_1x_3x_5\\
\qquad
+(M_{34}+M_{52})x_1x_2x_4+(M_{35}+M_{42})x_1x_2x_5\\
\qquad
+
M_{12}x_2^2x_3
+
M_{12}x_2x_4x_5
+
M_{13}x_2x_3^2
+
M_{13}x_3x_4x_5
+
M_{14}x_2x_3x_4
+
M_{14}x_4^2x_5\\
\qquad+
M_{15}x_2x_3x_5
+
M_{15}x_4x_5^2
+
M_{21}x_1^2x_3
+
M_{23}x_1x_3^2
+
M_{31}x_1^2x_2
+
M_{32}x_1x_2^2\\
\qquad
+
M_{41}x_1^2x_5
+
M_{45}x_1x_5^2
+
M_{51}x_1^2x_4
+
M_{54}x_1x_4^2.
\end{array}
\]
Thus, solving the equation $\innV{d\rho(M)x}{\nabla_xp(x)}=0$,
we need have all coefficients of each monomial must be equal to zero, 
and hence
$\mathfrak{g}[p]$ consists of matrices of the form
\[
t\cdot I_5+\pmat{
M_{11}&0&0&0&0\\
0&M_{22}&0&M_{24}&M_{25}\\
0&0&M_{33}&M_{34}&M_{35}\\
0&-M_{35}&-M_{25}&M_{44}&0\\
0&-M_{34}&-M_{24}&0&M_{55}
},
\quad
\begin{array}{l}
M_{11}+M_{22}+M_{33}=0,\\
M_{11}+M_{44}+M_{55}=0
\end{array}
\]
where $t\in\C$ and $M_{ij}\in \C$.
Here,
$I_5$ is the identity matrix of size $5$.
\end{Example}

As in Example \ref{exam:lie},
it is important to find out how many times each monomial $x_ax_jx_k$ appears.
In what follows.
we investigate it for $\mathtt{T}$ without edge sharing in detail enough to prove our main theorem.

Let $\mathtt{T}$ be a triangle arrangement with $n$ vertices.
Suppose that $\mathtt{T}$ has no edge sharing.

Fix a vertex $i$.
Then, a monomial $x_ax_jx_k$ appears for each triangle $\{i,j,k\}\in\mathcal{T}(i)$
and for each $a\in I_n$.
Let us find out from which vertex the monomial $x_ax_jx_k$ appears.

At first,
there are no duplicate in terms which are given from the vertex $i$.
In fact,
let us suppose that the monomial $x_ax_jx_k$ appears from triangles in $\mathcal{T}(i)$ other than $\{i,j,k\}$.
Then, at least one of $\{i,a,j\}$ and $\{i,a,k\}$ is included in $\mathcal{T}(i)$
because two of $x_a$, $x_j$ and $x_k$ come from partial derivative of $p(x)$.
In this case, however,  $\mathtt{T}$ has edge sharing, which leads to a contradiction.

Next, assume that the monomial $x_ax_jx_k$ comes from a vertex $l\in I_n$ such that $l\ne i$.
Then, we see that
at least one of $\{a,j,l\}\in\mathcal{T}$ and $\{a,k,l\}\in\mathcal{T}$ need satisfy
by the same reason to the case $i$.
In this situation, let us consider the positions of $i$, $a$ and $l$ in the graph consisting of edges of all triangles in $\mathtt{T}$.
Let us denote by $d_{\rm graph}$ the graph distance.
If $d_{\rm graph}(i,a)\ge 3$,
then there are no vertex $l$ satisfies $\{a,j,l\}\in\mathcal{T}$ or $\{a,k,l\}\in\mathcal{T}$
and hence the monomial $x_ax_jx_k$ never appear from vertices other than $i$.
This means that $M_{ia}=0$ whenever $d_{\rm graph}(i,a)\ge 3$,
or equivalently,
$M_{ia}\ne 0$ occurs only if $d_{\rm graph}(i,a)\le 2$.

\begin{enumerate}
\item[(0)] The case $d_{\rm graph}(i,a)=0$, that is, $a=i$.
In this case,
we have by setting $a=i$ in \eqref{eq:expand}
\[
\sum_{i=1}^nM_{ii}x_i\sum_{\{i,j,k\}\in\mathcal{T}(i)}x_jx_k
=
\sum_{\{i,j,k\}\in\mathcal{T}}(M_{ii}+M_{jj}+M_{kk})x_ix_jx_k,
\]
which leads to the following conditions.
\[
M_{ii}+M_{jj}+M_{kk}=0\quad\text{if }\{i,j,k\}\in\mathcal{T}.
\]
\item[(1)]
The case $d_{\rm graph}(i,a)=1$.
Assume that $\{i,j,k\},\,\{i,s,t\}\in\mathcal{T}(i)$ with $\{j,k\}\ne\{s,t\}$.
Suppose $a=j$.
Then, we have $x_ax_jx_k=x_j^2x_k$.
If the monomial $x_j^2x_k$ arises from the other vertex $l$,
then the triangle $\{j,k,l\}$ must included in $\mathcal{T}$,
but then $\mathcal{T}$ has edge sharing at the edge $jk$.
It creates a contradiction.
Next, we suppose that $a=s$.
Although the monomial $x_ax_jx_k=x_sx_jx_k$ has no information,
we know that $\partial_ip(x)$ have a monomial $x_sx_t$ so that
$x_ax_sx_t=x_s^2x_t$ appears in \eqref{eq:expand}.
Then, by the same reason on the case $a=j$, we face a contradiction and hence
we conclude that, if $\mathtt{T}$ has no edge sharing, then $M_{ia}=0$ whenever $d_{\rm graph}(i,a)=1$.

\item[(2)]
The case $d_{\rm graph}(i,a)=2$.
Let $\{j,a,b\}\in\mathcal{T}$, that is,
the vertex $a$ is linked to $j$.
The assumption that there are no edge sharing implies $k\not\in\{a,b\}$.

(2-i)
At first, suppose that the vertex $i$  is an isolated vertex.
In this case,
the partial derivative $\partial_ip(x)$ is exactly a monomial $x_jx_k$
so that $x_a\partial_ip(x)=x_ax_jx_k$,
and its factor $x_ax_j$ appears as a monomial in $\partial_bp(x)$.
Hence, if $\sharp\mathcal{T}(b)=1$, then we obtain
\[
x_ax_jx_k=x_a\partial_ip(x)=x_k\partial_bp(x)
\]
and this monomial never arises from the other vertices so that
we obtain a condition
\begin{equation}\label{eq:key eq}
M_{ia}+M_{bk}=0.
\end{equation}
Assume that $\sharp\mathcal{T}(b)\ge 2$.
If there are no $T\in\mathcal{T}(b)$ such that $k\in T$,
then we need have $M_{bk}=0$ and hence we also have $M_{ia}=0$.
For the case that there exists $T\in\mathcal{T}(b)$ such that $k\in T$,
we do not discuss in detail and only give one remark as follows.
If we set $T=\{k,b,c\}$,
then $c\not\in\{i,j,k,a,b\}$ because there are no edge sharing
and hence 
the following three triangles $\{i,j,k\}$, $\{j,a,b\}$ and $\{k,b,c\}$ form a ring of triangles $\mathtt{T}_B$
as in Example~\ref{exam:TrigArr}.

(2-ii)
Next, suppose that the vertex $i$ is not an isolated vertex.
Let $\{i,j,k\}$, $\{i,s,t\}\in \mathcal{T}(i)$ with $\{j,k\}\cap\{s,t\}=\emptyset$.
In this case,
polynomial $x_a\partial_ip(x)$ includes two monomials $x_ax_jx_k$ and $x_ax_sx_t$.
If the vertex $a$ is not linked to $s$ and $t$,
then 
the monomials $x_ax_s$ and $x_ax_t$ cannot appear when we differentiate $p(x)$
so that we obtain $M_{ia}=0$.
For the case that the vertex $a$ is linked to $s$ or $t$,
then we do not discuss in detail and only give one remark as follows.
If we set $\{a,c,t\}\in\mathcal{T}$,
then $c\not\in\{a,b,i,j,s,t\}$ because there are no edge sharing ($c=k$ may occur),
so that $\mathtt{T}$ must include at least one of the following triangle arrangements.
\begin{center}
\begin{tabular}{c@{\hspace{3em}}c}
\includegraphics[scale=0.5]{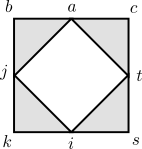}
&
\includegraphics[scale=0.7]{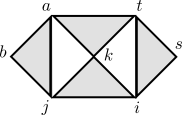}
\end{tabular}
\end{center}
\end{enumerate}

\section{Attaching two triangle arrangements at a vertex}
\label{sect:main}

In this section,
we present our main theorem stating that
we are able to construct a prehomogeneous triangle arrangement
by attaching two prehomogeneous triangle arrangements at a vertex.

\begin{Theorem}
\label{theo}
Let $\mathtt{T}_\nu$ $(\nu=1,2)$ be two prehomogeneous triangle arrangements with no edge sharing
and let $(\mathfrak{g}_\nu,d\rho_\nu,V_\nu)$ be the corresponding prehomogeneous vector spaces.
Suppose that
there exist vertices $0^{(\nu)}$ of $\mathtt{T}_\nu$ and subalgebras $\mathfrak{h}_\nu$ of $\mathfrak{g}_\nu$ $(\nu=1,2)$ such that
\begin{center}
\begin{tabular}{cp{.85\textwidth}}
{\rm(1)}&triplets $(\mathfrak{h}_\nu, d\rho_\nu|_{\mathfrak{h}_\nu},V_\nu)$ are prehomogeneous vector spaces,\\
{\rm(2)}&variables $x_{0^{(\nu)}}$ corrresponding to the vertices $0^{(\nu)}$ are relatively invariant under the actions of $\mathfrak{h}_\nu$,\\
{\rm(3)}&for each $\nu=1,2$,
there exists at least one triangle $\{0^{(\nu)},a^{(\nu)},\bar{a}^{(\nu)}\}\in\mathcal{T}(0^{(\nu)})$
such that $a^{(\nu)}$ is an isolated vertex and the variable $x_{\bar{a}^{(\nu)}}$ is relatively invariant under the action of $\mathfrak{h}_\nu$.
\end{tabular}
\end{center}
Then, 
the triangle arrangement $\mathtt{T}$ obtained by attaching two triangle arrangements $\mathtt{T}_\nu$ at vertices $0^{(\nu)}$ is prehomogeneous.
\end{Theorem}

\begin{Example}
\label{exam:attach}
Let $\mathtt{T}_1=\mathtt{T}_A$ and $\mathtt{T}_2=\mathtt{T}_D$ as in Examples \ref{exam:TrigArr} and \ref{exam:lie}.
Let $\mathtt{T}$ be the triangle arrangement obtained by attaching $\mathtt{T}_\nu$ $(\nu=1,2)$
at vertices $0^{(1)}=5$ in $\mathtt{T}_1$ and $0^{(2)}=2$ in $\mathtt{T}_2$.
Then, $\mathtt{T}$ is drawn as follows.
\begin{center}
\begin{tabular}{c@{\hspace{1.5em}}c@{\hspace{1.5em}}c@{\hspace{1.5em}}c@{\hspace{1.5em}}c}
$\mathtt{T}_1$&$+$&$\mathtt{T}_2$&$\to$&$\mathtt{T}$\\
\includegraphics[scale=0.5]{pic_preprint/TrigArr1.png}
&\raisebox{1.5em}{$+$}&
\includegraphics[scale=0.5]{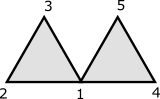}
&
\raisebox{1.5em}{$\to$}&
\raisebox{-2em}{\includegraphics[scale=0.5]{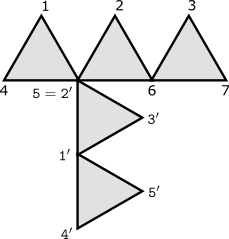}}
\end{tabular}
\end{center}
In the figure of $\mathtt{T}$,
we attach the symbol prime ${}^\prime$ to vertices coming from $\mathtt{T}_2$ 
in order to distinguish those from $\mathtt{T}_1$.
Although the polynomial $x_2$ is not a relative invariant with respect to
the Lie algebra $\mathfrak{g}_2=\mathfrak{g}_D$,
a subalgebra $\mathfrak{h}_2=\set{M\in\mathfrak{g}_2}{M_{24}=M_{25}=0}$ acts on $V_2=\C^5$ prehomogeneously
and 
$x_2$ is a relatively invariant polynomial with respect to $\mathfrak{h}_2$
so that we can apply Theorem~\ref{theo}.
Thus, we see that $\mathtt{T}$ is also prehomogeneous.
\end{Example}

\begin{proof}[Proof of Theorem $\ref{theo}$]
Set $\dim V_\nu=n_\nu+1$.
Then, vertices of $\mathtt{T}_\nu$ are $0^{(\nu)}$, $1^{(\nu)}$, \dots, $n_{\nu}^{(\nu)}$.
For a basis of $V_\nu$, we choose the standard basis $\bs{e}_{0^{(\nu)}},\bs{e}_{1^{(\nu)}},\dots,\bs{e}_{n_\nu^{(\nu)}}$.
Since a polynomial $x_{0^{(\nu)}}$ is relatively invariant for each $\nu=1,2$,
a general element $M^{(\nu)}$ of $\mathfrak{h}_\nu$ is described as
\begin{equation}
\label{eq:elh}
M^{(\nu)}=
\pmat{m^{(\nu)}&0\\h^{(\nu)}&\widetilde{M}^{(\nu)}}
\quad
\bigl(m^{(\nu)}\in\C,\ h^{(\nu)}\in\C^{n_\nu},\ 
\widetilde{M}^{(\nu)}\in\mathrm{Mat}(n_\nu,\C)\bigr).
\end{equation}
Note that $m^{(\nu)}$, $h^{(\nu)}$ and $\widetilde{M}^{(\nu)}$ may be depend on each other.

Let $\mathtt{T}$ be the triangle arrangement which is obtained by attaching $\mathtt{T}_\nu$ $(\nu=1,2)$ at the points $0^{(\nu)}$.
The corresponding vector space $V$ has $\dim V=n_1+n_2+1$.
Vertices of $\mathtt{T}$ are labelled as 
\[
0=0^{(1)}=0^{(2)},\quad
i=i^{(1)}\quad(i=1,\dots,n_1),\quad
n_1+j=j^{(2)}\quad(j=1,\dots,n_2).
\]
Put $\mathcal{V}:=\{0,1,\dots,n_1+n_2\}$.
We denote by $\mathfrak{g}$ the Lie algebra corresponding to $\mathtt{T}$.

The proof is separated into two part, one is determining a  structure of $\mathfrak{g}$,
and the other is proving prehomogeneity.

\noindent(i)
We first investigate a structure of $\mathfrak{g}$.
Since $\mathtt{T}$ obviously has no edge sharing,
for an element $M=(M_{ij})\in\mathfrak{g}$,
a condition $M_{ij}$ occur only if $d_{\rm graph}(i,j)$ for $i,j\in\mathcal{V}$.
If $i,j$ are included in one $\mathtt{T}_\nu$,
then it relates to $\mathfrak{g}_\nu$ and thus we do not consider again.
Hence,
it is enough to consider the case that one of $i,j$ is included in $\mathtt{T}_1$ and the other one is in $\mathtt{T}_2$.
Since $\mathtt{T}_1$ and $\mathtt{T}_2$ are joined at one point,
such a pair $i,j$ must satisfy $d_{\rm graph}(0,i)=d_{\rm graph}(0,j)=1$.
In what follows,
we use a symbol $i$ for vertex in $\mathtt{T}_1$ and a symbol $a$ (instead of $j$) for vertex in $\mathtt{T}_2$
in order to distinguish $\mathtt{T}_\nu$ easily by symbols.
Triangles in $\mathtt{T}(0)$ are written like $\{0,a,\bar{a}\}$,
that is,
we use bar symbol $\bar{a}$ for the remaining vertex.

The current situation is included in a situation discussed in the previous section by setting $j=0$.
Since $\mathtt{T}_1$ and $\mathtt{T}_2$ are joined at one point $0$,
rings consisting of triangles, which are excluded in the discussion of the previous section, never appear.
Thus, \eqref{eq:key eq} is the only non-trivial relation which implies
for each pair of $\{0^{(1)},i,\bar{i}\}\in\mathcal{T}_1(0^{(1)})$ and $\{0^{(2)},a,\bar{a}\}\in\mathcal{T}_2(0^{(2)})$,
we have
\begin{equation}
\label{eq:sum}
M_{i\bar{a}}+M_{a\bar{i}}=0,
\end{equation}
and $M_{ia}=0$ or $M_{ai}=0$ otherwise.
By discussion (2) in the previous section,
an element in \eqref{eq:sum} do not vanish if and only if
$i$ is an isolated vertex in $\mathtt{T}_1$ and $a$ is an isolated vertex in $\mathtt{T}_2$.
The assumption (3) ensures existence of such vertices $i,a$, that is,
there exists at least one $\{0^{(1)},i,\bar{i}\}\in\mathcal{T}_1(0^{(1)})$ such that
$i$ is an isolated vertex and a polynomial $x_{\bar{i}}$ is relatively invariant with respect to $\mathfrak{h}_1$,
and similarly
there exists at least one $\{0^{(2)},a,\bar{a}\}\in\mathcal{T}_2(0^{(2)})$ such that
$a$ is an isolated vertex and a polynomial $x_{\bar{a}}$ is relatively invariant with respect to $\mathfrak{h}_2$.

Therefore, we have confirmed that
$\mathfrak{g}$ includes a subalgebra $\mathfrak{h}$ consisting of matrices of the form
\begin{equation}
\label{eq:el h}
M=\pmat{
m&0&0\\
h^{(1)}&M^{(1)}&Z_1\\
h^{(2)}&Z_2&M^{(2)}
},
\end{equation}
where $m=m^{(1)}=m^{(2)}\in\C$,
$h^{(\nu)}\in\C^{n_\nu}$ and $\widetilde{M}^{(\nu)}\in\mathrm{Mat}(n_\nu,\C)$ are as in \eqref{eq:elh}.
For matrices $Z_1=(M_{i\bar{a}})_{1\le i\le n_1,\ 1\le a\le n_2}$
and
$Z_2=(M_{a\bar{i}})_{1\le a\le n_2,\ 1\le i\le n_1}$,
elements $M_{i\bar{a}}$ and $M_{i\bar{a}}$ are zeros except for the case
\begin{center}
\begin{tabular}{cl}
(a)&$\{0^{(1)},i,\bar{i}\}\in\mathcal{T}_1(0^{(1)})\text{ and }\{0^{(2)},a,\bar{a}\}\in\mathcal{T}_2(0^{(2)})$,\\
(b)&the vertex $i$ is an isolated vertex in $\mathtt{T}_1$,\\
(c)&a polynomial $x_{\bar{a}}$ is relatively invariant with respect to $\mathfrak{h}_2$,
\end{tabular}
\end{center}
and in this case we have
\begin{equation}
\label{eq:rel}
M_{a\bar{i}}=-M_{i\bar{a}}.
\end{equation}

\noindent(ii) 
Next we investigate the prehomogeneity of $\mathtt{T}$.
To do so, we consider basic relative invariants with respect to $\mathfrak{h}$.
We set
\[
\begin{array}{l}
p_0(x)=p_0(x_0,x^{(1)},x^{(2)})=p^{(1)}(x_0,x^{(1)})+p^{(2)}(x_0,x^{(2)}),\\
p_1(x)=p_1(x_0,x^{(1)},x^{(2)})=x_0.
\end{array}
\]
It is obvious that $p_0(x)$ and $p_1(x)$ are relatively invariant polynomials.
For $\nu=1,2$,
let us denote by $p^{(\nu)}(x_{0^{(\nu)}},x^{(\nu)})$ and $q_j^{(\nu)}(x_{0^{(\nu)}},\,x^{(\nu)})$ $(j=0,1,\dots,k_\nu)$ the basic relative invariants with respect to $\mathfrak{h}_\nu$.
Here, $p^{(\nu)}(x_{0^{(\nu)}},x^{(\nu)})$ is the polynomial corresponding to triangle arrangement $\mathtt{T}_{\nu}$
and $q_j^{(\nu)}(x_{0^{(\nu)}},\,x^{(\nu)})$ are the other ones.
We set $q_0^{(\nu)}(x_{0^{(\nu)}},\,x^{(\nu)})=x_{0^{(\nu)}}$.

Among polynomials $q_j^{(\nu)}(x_{0^{(\nu)}},\,x^{(\nu)})$ $(j=1,\dots,k_\nu)$,
we pick ones such that $\partial_jq_i^{(\nu)}(x_{0^{(\nu)}},\,x^{(\nu)})=0$ for any isolated vertex $j$ in triangles in $\mathcal{T}_\nu(0^{(\nu)})$,
and rename them as $p_2(x),\dots,p_k(x)$.

We shall show that
the basic relative invariants with respect to $\mathfrak{h}$ are exactly $p_j(x)$ $(j=0,1,\dots,k)$,
and the prehomogeneity is proved in the same time.

Let $H=\exp \mathfrak{h}$ be a connected and simply connected Lie group of $\mathfrak{h}$.
Let us take a reference point $x_*\in V$ such that $p_j(x_*)\ne 0$ for all $j=0,1,\dots,k$.
What we want to prove is to show that any regular element $x\in V$ , that is, $p_j(x)\ne 0$ for any $j=0,1,\dots,k$
can be moved to $x_*$ by the action of $H$.
To do so,
we decompose $\mathfrak{h}$ into $4$ spaces as a vector space.
At first, we put
\[
\mathfrak{h}'=\set{M\in\mathfrak{h}}{\innV{d\rho(M)x}{\nabla_xp_0(x)}=0\text{ for all }x\in V}
\]
and 
\[
\mathfrak{h}''=\set{M'\in \mathfrak{h}'}{d\rho(M')x_0=0}\subset \mathfrak{h}'.
\]
Then, $\mathfrak{h}''$ is a subalgebra of $\mathfrak{h}'$ so that
there exists an $A\in\mathfrak{h}'$ such that
\[
\mathfrak{h}'=\C A\,\dot+\,\mathfrak{h}''.
\]
Note that we can take $A$ as a diagonal matrix.
Then, an element $M$ as in \eqref{eq:el h} can be decomposed into
\begin{equation}
\label{eq:decomp}
M=t\,I+m' A
+
\pmat{0&0&0\\0&M_1'&0\\0&0&M_2'}
+
\pmat{0&0&0\\0&0&Z_1\\0&Z_2&0}.
\end{equation}
Here, $I$ is the identity matrix of suitable size.
Set 
\[
M'=\pmat{0&0&0\\0&M_1'&0\\0&0&M_2'},\quad
Z:=\pmat{0&0&0\\0&0&Z_1\\0&Z_2&0}.
\]
Then, we have $M',Z\in\mathfrak{h}''$.
By definition,
the polynomial $p_0(x)$ is invariant under the actions of $A$, $M'$ and $Z$,
and the polynomial $p_1(x)$  is invariant under the actions of $M'$ and $Z$.
This means that
we can match a value of $p_0(x)$ to $p_0(x_*)$ by the action $\exp(t\,I)$,
and then a value of $p_1(x)$ to $p_1(x_*)$ by the action of $\exp(m'A)$.

If we have $q_j^{(\nu)}(x_{0^{(\nu)}},\,x^{(\nu)})\ne 0$ for any $j=0,1,\dots,k_\nu$ and for $\nu=1,2$,
then the prehomogeneity of $\mathfrak{h}_\nu$ imply that of $\mathfrak{h}$, and hence of $\mathfrak{g}$.
The problem here is that there are some $q_j^{(\nu)}(x_{0^{(\nu)}},\,x^{(\nu)})$ 
which are not included in $p_j(x)$ $(j=0,1,\dots,k)$.
Such $q_j^{(\nu)}(x_{0^{(\nu)}},\,x^{(\nu)})$ satisfy
$\partial_i q_j^{(\nu)}(x_{0^{(\nu)}},\,x^{(\nu)})\ne 0$ for some isolated vertex $i$ in a triangle in $\mathcal{T}_\nu(0^{(\nu)})$.
By \eqref{eq:rel},
we can take $Z=E_{i\bar{a}}-E_{a\bar{i}}$ where $E_{st}$ is the matrix unit of size $\dim V$ having one at the position $(s,t)$ and zeros elsewhere.
Since $\exp Z=I+Z$, 
polynomials $q_j^{(\nu)}(x_{0^{(\nu)}},\,x^{(\nu)})$ which are not included in $p_j(x)$ can be take non-zero values
by applying actions of $\exp Z$.
Therefore,
we can conclude that
any regular element $x\in V$ can be moved to the reference point $x_*\in V$ by the action of $H$.
Namely,
we have proved that $H$ acts transitively on the set of regular elements and hence 
$(H,\rho,V)$ is a prehomogeneous vector space and so is $(G,\rho,V)$.
\end{proof}

\begin{Remark}
The condition (3) in Theorem~\ref{theo} is necessary.
We shall confirm this by the following example.
\begin{center}
\begin{tabular}{c@{\hspace{1.5em}}c@{\hspace{1.5em}}c@{\hspace{1.5em}}c@{\hspace{1.5em}}c}
$\mathtt{T}_1$&$+$&$\mathtt{T}_2$&$\to$&$\mathtt{T}$\\
\raisebox{-1em}{\includegraphics[scale=0.5]{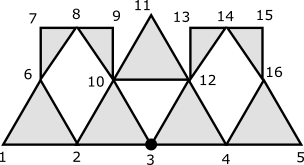}}
&\raisebox{1.5em}{$+$}
&
\includegraphics[scale=0.5]{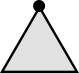}
&\raisebox{1.5em}{$\to$}&
\raisebox{-2.3em}{\includegraphics[scale=0.5]{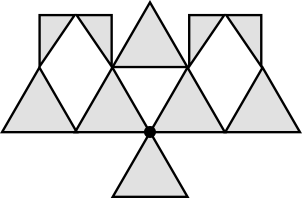}}
\end{tabular}
\end{center}
Both triangle arrangements $\mathtt{T}_\nu$ $(\nu=1,2)$ above are prehomogeneous,
and variables corresponding to the vertices in black circles in each figures are relatively invariant polynomials,
but there are no triangles in $\mathcal{T}_1(3)$ including isolated points.
In this case,
the triangle arrangement $\mathtt{T}$ obtained by attaching $\mathtt{T}_\nu$ $(\nu=1,2)$ at the vertices of black circles is not prehomogeneous.
We note that if we choose a vertex from one of $\{6,8,10,12,14,16\}$ for the attaching point in $\mathtt{T}_1$,
then the condition (3) is satisfied so that
the triangle arrangement obtained by attaching this point is prehomogeneous.
\end{Remark}

\section{Examples}
\label{sect:example}

In this section,
we give some series of prehomogeneous triangle arrangements.
We also exhibit triangle arrangements which does not correspond to prehomogeneous vector spaces.

\begin{Theorem}
\label{theo:pv}
The following triangle arrangements are prehomogeneous.
\begin{center}
\begin{tabular}{c@{\ }p{.9\textwidth}}
$(1)$&daisy cases: triangle arrangements constructed by attaching $n$ triangles at one vertex $(n\ge 2)$,\\
$(2)$&chain cases: triangle arrangements constructed by arraying $n$ triangles in a row $(n\ge 2)$,\\
$(3)$&circular cases: triangle arrangements constructed by arraying $n$ triangles circularly $(n\ge 3)$,\\
$(4)$&edge gluing cases: triangle arrangements constructed by gluing $n$ triangle arrangements $\mathtt{T}_B$ $($as in Example $\ref{exam:TrigArr})$  edges $24$ and $36$ $(n\ge 2)$.
\end{tabular}
\end{center}
\end{Theorem}

Examples of triangle arrangements in Theorem~\ref{theo:pv} are given in Figure~\ref{fig:general}.
In particular,
Theorem~\ref{theo:pv} (4) tells us that
the condition of edge-sharing is not a necessary condition for prehomogeneity.

\begin{figure}[ht]
\centering
\begin{tabular}{c@{\hspace{3em}}c@{\hspace{3em}}c}
(1) Daisy case&
(2) Chain case&
(3) Circular case\\
($n=6$)&($n=4$)&($n=6$)\\
\includegraphics[scale=0.4]{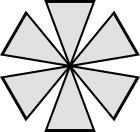}&
\raisebox{1em}{\includegraphics[scale=0.4]{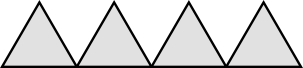}}&
\raisebox{-.5em}{\includegraphics[scale=0.4]{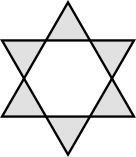}}
\end{tabular}

\vspace{1em}

\begin{tabular}{c@{\hspace{3em}}c}
(4) Edge gluing case&
(4) Edge gluing case\\
($n=2$)&($n=3$)\\
\includegraphics[scale=0.4]{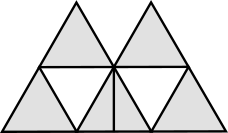}
&
\includegraphics[scale=0.4]{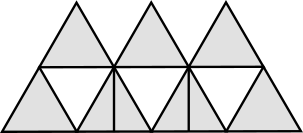}
\end{tabular}
\caption{Examples of triangle arrangements of Theorem~\ref{theo:pv}}
\label{fig:general}
\end{figure}

We shall prove this theorem in the following subsections by giving detailed structures of Lie algebras.
We note here that, in this section, the dual vector space $V^*$ be identified with $V$ through $\innV{\cdot}{\cdot}$.
Before going to proofs,
we give an example of triangle arrangements which are not prehomogeneous.

\begin{Proposition}
\label{prop:unconnected}
Triangle arrangements, which are not connected as graphs, are not prehomogeneous.
\end{Proposition}
\begin{proof}
If a triangle arrangement $\mathtt{T}$ is not connected,
then there exist $\mathtt{T}_1$ and $\mathtt{T}_2$ such that
$\mathtt{T}=\mathtt{T}_1\cup\mathtt{T}_2$
and
$\mathtt{T}_1\cap\mathtt{T}_2=\emptyset$.
Let $p(x)$ and $q(y)$ be the corresponding polynomials associated with 
$\mathtt{T}_1$ and $\mathtt{T}_2$, respectively.
Then, it is obvious that
the polynomial corresponding to $\mathtt{T}$ is $P(x,y)=p(x)+q(y)$.
Then, since there is no overlapping at variables of $p$ and $q$,
the corresponding Lie algebra $\mathfrak{g}[P]$ can be described as
\[
\mathfrak{g}[P]=\set{tI+\pmat{M_1&0\\0&M_2}}{t\in\C,\,M_1\in\mathfrak{g}_0[p],\,M_2\in\mathfrak{g}_0[q]},
\]
where $\mathfrak{g}[p]$, $\mathfrak{g}[q]$ are the Lie algebras corresponding to $p$ and $q$, respectively.
Thus, it is easily verified that
if one of $\mathfrak{g}[p]$ or $\mathfrak{g}[q]$ is not prehomogeneous,
then so is not $\mathfrak{g}[P]$.
Therefore,
we can assume that both of them are prehomogeneous,
and
we shall prove the assertion for this case for a more general situation.
To do so,
we recall the definition of \textit{homaloidal} polynomials.
A homogeneous polynomial $p(x)$ of degree $d$ is said to be homaloidal
if the polar map $\varphi_p(x):=\mathrm{grad}\,\log p(x)$ is birational.
Then, there exists a homogeneous polynomial $p_*(x)$,
called the dual polynomial of $p(x)$, such that
$p_*(\nabla_xp(x))=p(x)^{d-1}$.
For a relatively invariant polynomial $p(x)$ of a prehomogeneous vector space,
it is known that the polar map $\varphi_p(x)$ is homaloidal or a zero map.
Thus, it is enough to prove the following claim.
\end{proof}

\begin{Claim}
Let $d\ge 3$.
For two homaloidal polynomials $p(x)=p(x_1,\dots,x_n)$ and $q(y)=q(y_1,\dots,y_m)$ of degree $d$,
we set $P(x,y):=p(x)+q(y)$.
Then, $P(x,y)$ cannot be a homaloidal polynomial.
\end{Claim}
\begin{proof}
Let $p_*(x)$ and $q_*(y)$ be the dual polynomials of $p(x)$ and $q(y)$, respectively, that is, we have
\[
p_*\bigl(\nabla_xp(x)\bigr)=p(x)^{d-1},\quad
q_*(\nabla_yq(y)\bigr)=q(y)^{d-1}.
\]
Put
\[
P_*(x,y):=\bigl(p_*(x)^{\frac{1}{d-1}}+q_*(y)^{\frac{1}{d-1}}\bigr)^{d-1}.
\]
Then,
we have
\[
\begin{array}{r@{\ }c@{\ }l}
P_*\bigl(\nabla_{x,y}P(x,y)\bigr)
&=&
\ds
P_*\bigl(\nabla_xp(x),\,\nabla_y(y)\bigr)\\
&=&
\ds
\Bigl(\bigl\{p_*(\nabla_xp(x))\bigr\}^{\frac{1}{d-1}}+\bigl\{q_*(\nabla_yq(y))\bigr\}^{\frac{1}{d-1}}\Bigr)^{d-1}\\
&=&
\Bigl(p(x)+q(y)\Bigr)^{d-1}
=
P(x,y)^{d-1}.
\end{array}
\]
Since $P_*(x,y)$ is not a rational function,
we see that the polar map $\varphi_P(x,y)$ cannot be birational so that
$P(x,y)$ cannot be a homaloidal polynomial.
\end{proof}

\subsection{Daisy cases}

Let $\mathtt{T}$ be a triangle arrangement which is constructed by attaching $n$ triangles at one vertex for $n\ge 2$
(see (1) of Figure ~\ref{fig:general}).
In this case,
the number of vertices are $2n+1$, that is, $\dim V=2n+1$.
The $i$-th triangle consists of vertices $\{i,i+1,2n+1\}$.
The corresponding polynomial $p(x)$ is
\[
p(x)=(x_1x_2+x_3x_4+\cdots+x_{2n-1}x_{2n})x_{2n+1}.
\]
Obviously, $p(x)$ is obtained as a product of polynomials of degree one and two,
which do not share variables, so that 
the corresponding triplet $(\mathfrak{g}[p],d\rho,V)$ is a prehomogeneous vector space.
Let $J$ be a $2n \times 2n$ matrix defined by
\[
J=\mathrm{diag}(J',\dots,J'),\quad
J'=\pmat{0&1\\1&0}.
\]

\begin{Lemma}
The prehomogeneous vector space $(\mathfrak{g}[p],d\rho,V)$ is regular and reductive.
A general element $M$ of $\mathfrak{g}[p]$ is of the form
\[
M=\pmat{tI_{2n}+M'&0\\0&M_{2n+1}},\quad
(t,M_{2n+1}\in\C,\ M'\in\mathfrak{so}(J)),
\]
and hence one has $\dim\mathfrak{g}[p]=2n^2-n+2$.
The basic relative invariants are given as
\[
p_0(x)=p(x),\quad
p_1(x)=x_{2n+1}\quad(x\in V).
\]
\end{Lemma}
\begin{proof}
It is enough to mention that $\mathfrak{so}(J)$ is a Lie algebra with respect to the following bilinear form
\[
x_1x_2+x_3x_4+\cdots+x_{2n-1}x_{2n}
=
\frac12\innV{x'}{Jx'}\quad(x'=(x_1,x_2,\dots,x_{2n})),
\]
and we have $p(x)=\frac12\innV{x'}{Jx'}\cdot x_{2n+1}$.
\end{proof}

\subsection{Chain cases}

Let $\mathtt{T}$ be a triangle arrangement which is constructed by arraying $n$ triangles in a row with $n\ge 3$
(see (2) in Figure~\ref{fig:general}).
In this case, the number of vertices are $2n+1$, that is, $\dim V=2n+1$.
The $i$-th triangle consists of vertices $\{i,\,n+i,\,n+i+1\}$.
The corresponding polynomial $p(x)$ is given as
\[
p(x)=\sum_{i=1}^nx_ix_{n+i}x_{n+i+1}\quad(x\in V).
\]
Let $E_{ij}^{(m)}$ (resp.\ $E'_{ij}$) be a matrix unit of size $m\times m$ (resp.\ $n\times (n+1)$) having one on the position $(i,j)$ and zeros elsewhere.

\begin{Lemma}
A general element $M$ of $\mathfrak{g}[p]$ is of the form
\[
M=\pmat{M_{11}&M_{12}\\0&M_{22}}
\]
where
\[
\begin{cases}
\ds
M_{11}=
bE_{21}^{(n)}+cE_{n-1,n}^{(n)}
+
\sum_{i=1}^n(t-t_i-t_{i+1})E_{ii}^{(n)},\\
\ds
M_{12}=\sum_{i=1}^{n-1}a_i(E'_{i+1,i}-E'_{i,i+2}),\\
\ds
M_{22}=
-bE_{13}^{(n+1)}-cE_{n+1,n-1}^{(n+1)}
+
\sum_{i=1}^{n+1}t_iE_{ii}^{(n+1)}.
\end{cases}
\]
\end{Lemma}

\begin{proof}
Since the triangle arrangement $\mathtt{T}$ does not have edge sharing,
we can apply the discussion in Section \ref{sect:structure}.
Let $j=1,\dots,n$.
Then, the vertex $j$ is isolated, and the equation \eqref{eq:key eq} can be written as
\[
M_{j,n+j+2}+M_{j+1,n+j}=0\quad(j=1,\dots,n).
\]
Moreover, the vertices $i=n+1$ and $i=2n+1$ are also isolated so that
the equation \eqref{eq:key eq} again implies 
\[
M_{n+1,n+3}+M_{2,n+1}=0,\quad
M_{2n+1,2n-1}+M_{n-1,n}=0.
\]
The other terms are all zeros so that
the proof is now completed.
\end{proof}

This lemma yields that
$\mathfrak{g}[p]$ is a solvable Lie algebra of $\dim\mathfrak{g}[p]=2n+3$.

\begin{Lemma}
The triplet $(\mathfrak{g}[p],d\rho,V)$ is a regular prehomogeneous vector space for all $n\ge 3$.
Its basic relative invariants $p_i(x)$ $(i=1,\dots,n)$ are given as
\[
p_1(x)=p(x),\quad
p_i(x)=x_{n+i}\quad(i=2,\dots,n).
\]
\end{Lemma}
\begin{proof}
We use Lemma~\ref{lemma}.
To make discussion simple,
we consider a subalgebra $\mathfrak{h}$ of $\mathfrak{g}[p]$ defined by
\[\mathfrak{h}:=\set{M\in\mathfrak{g}[p]}{b=c=0},\]
where we use an expression of $M\in\mathfrak{g}[p]$ as in the above lemma.
For $x\in V$,
let $A(x)\colon \mathfrak{g}[p]\to V$ be a linear map defined by $A(x)M:=d\rho(M)x$ $(M\in\mathfrak{h})$.
Then, since $\dim\mathfrak{h}=\dim V$, we see that $A(x)$ is a square matrix and by the above lemma
{\tiny
\[
A(x)=
\left(
\begin{array}{c|cccc|cccccc}
x_1&-x_{n+3}&0&\cdots&0&-x_1&-x_1&0&\cdots&&0\\
x_2&x_{n+1}&-x_{n+4}&\ddots&\vdots&0&-x_2&-x_2&\ddots&&0\\
\vdots&0&x_{n+2}&\ddots&0&\vdots&\ddots&\ddots&\ddots&&\vdots\\
\vdots&\vdots&\ddots&\ddots&-x_{2n+1}&&&&&-x_{n-1}&0\\
x_n&0&\cdots&0&x_{2n-1}&0&\cdots&&0&-x_n&-x_n\\ \hline
0&&&&&x_{n+1}\\
0&&&&&&x_{n+2}\\
0&&&&&&&\ddots\\
0&&&&&&&&\ddots\\
0&&&&&&&&&x_{2n}\\
0&&&&&&&&&&x_{2n+1}
\end{array}\right).
\]
}
Thus, we can calculate its determinant by using cofactor expansion at the first column as
\[
\begin{array}{r@{\ }c@{\ }l}
\det A(x)
&=&
\ds
x_{n+1}\cdots x_{2n+1}\det\pmat{
x_1&-x_{n+3}&0&\cdots&0\\
x_2&x_{n+1}&-x_{n+4}&\ddots&\vdots\\
\vdots&0&x_{n+2}&\ddots&0\\
\vdots&\vdots&\ddots&\ddots&-x_{2n+1}\\
x_n&0&\cdots&0&x_{2n-1}
}\\
&=&
\ds
x_{n+1}\cdots x_{2n+1}
\sum_{i=1}^n
x_ix_{n+i}x_{n+i+1}\cdot x_{n+3}\cdots x_{2n-1}\\
&=&
\ds x_{n+1}x_{n+2}(x_{n+3}\cdots x_{2n-1})^2x_{2n}x_{2n+1}\cdot p(x).
\end{array}
\]
This shows that $A(x)$ has generic full rank and hence the triplet $(\mathfrak{g}[p],d\rho,V)$ is a prehomogeneous vector space.
This also shows that polynomials $x_{n+i}$ $(i=1,\dots,n+1)$ are basic relative invariants with respect to $\mathfrak{h}$.
Among them, two polynomials $x_{n+1}$ and $x_{2n+1}$ are not relatively invariant
because $\mathfrak{g}[p]$ has $b,c$ parts,
but the other ones are relatively invariant with respect to $\mathfrak{g}[p]$.
\end{proof}

We now calculate on the dual prehomogeneous vector space.
Let us consider a linear map $A^*(x)\colon \mathfrak{h}\to V$ $(x\in V)$ defined by
$A^*(x)M:={}^{t\!}d\rho(M)x$ $(M\in\mathfrak{h})$.
Then, $A^*(x)$ is a matrix of the form
{\tiny
\[
A^*(x)=
\left(
\begin{array}{c|cccccc|cccc}
x_1&-x_1&-x_1&&&&&&	\\
x_2&&-x_2&-x_2&&&&&&	\\
\vdots&&&\ddots&\ddots&&&&	\\
x_{n-1}&&&&-x_{n-1}&-x_{n-1}&&\\
x_n&&&&&-x_n&-x_n	\\ \hline
0&x_{n+1}&&&&&&	x_2\\
0&&x_{n+2}&&&&&	0&\ddots\\
\vdots&&&\ddots&&&&	-x_1&\ddots&\ddots\\
\vdots&&&&\ddots&&&	&\ddots&\ddots&x_n\\
0&&&&&x_{2n}&&	&&\ddots&0\\
0&&&&&&x_{2n+1}&	&&&-x_{n-1}\\
\end{array}
\right).
\]
}
Its determinant can be calculated as 
if $n=2k$ is even then
\[
\det A^*(x)=
x_1\cdots x_{n}\times \det
\pmat{
x_{n+1}&0&	x_2\\
0&x_{n+2}&	0&\ddots\\
x_{n+3}&0&	-x_1&\ddots&\ddots\\
\vdots&\vdots&&\ddots&\ddots&x_{n}	\\
0&x_{2n}&	&&-x_{n-2}&0\\
x_{2n+1}&0&	&&&-x_{n-1}
},
\]
and if $n=2k+1$ is odd, then
\[
\det A^*(x)=
x_1\cdots x_{n}\times \det
\pmat{
x_{n+1}&0&	x_2\\
0&x_{n+2}&	0&\ddots\\
x_{n+3}&0&	-x_1&\ddots&\ddots\\
\vdots&\vdots&&\ddots&\ddots&x_{n}	\\
x_{2n}&0&	&&-x_{n-2}&0\\
0&x_{2n+1}&	&&&-x_{n-1}
}.
\]
Thus,
the dual prehomogeneous vector space
has basic relative invariants $q_i(x)$ $(i=0,1,\dots,n-1)$ defined by
\[
q_i(x)=x_i\quad(i=2,\dots,n-1).
\]
The other ones $q_0(x)$ and $q_1(x)$ are defined according to $n$ is odd or even.
If $n$ is even, then 
\[
q_0(x)=\left|\begin{smallmatrix}
x_{n+1}&x_2\\
x_{n+3}&-x_1&x_4\\
\vdots&&\ddots&\ddots\\
\vdots&&&-x_{n-3}&x_{n}\\
x_{2n+1}&0&\cdots&0&-x_{n-1}
\end{smallmatrix}\right|,\quad
q_1(x)=\left|\begin{smallmatrix}
x_{n+2}&x_3\\
x_{n+4}&-x_2&x_5\\
\vdots&&\ddots&\ddots\\
\vdots&&&-x_{n-4}&-x_{n-1}\\
x_{2n}&0&\cdots&0&-x_{n-2}
\end{smallmatrix}\right|
\]
and if $n$ is odd, then
\[
q_0(x)=\left|\begin{smallmatrix}
x_{n+1}&x_2\\
x_{n+3}&-x_1&x_4\\
\vdots&&\ddots&\ddots\\
\vdots&&&-x_{n-4}&x_{n-1}\\
x_{2n}&0&\cdots&0&-x_{n-2}
\end{smallmatrix}\right|,\quad
q_1(x)=\left|\begin{smallmatrix}
x_{n+2}&x_3\\
x_{n+4}&-x_2&x_5\\
\vdots&&\ddots&\ddots\\
\vdots&&&-x_{n-3}&-x_{n}\\
x_{2n+1}&0&\cdots&0&-x_{n-1}
\end{smallmatrix}\right|
\]

\begin{Remark}
Let $n=3$.
In this case,
we have
\[
p(x)=p_1(x)=x_1x_4x_5+x_2x_5x_6+x_3x_6x_7\quad(x\in V),
\]
and the other basic relative invariants are $p_2(x)=x_5$ and $p_3(x)=x_6$.
Let us change variables as follows.
\[
\begin{array}{c}
\ds
z_{11}=x_2,\quad
z_{22}=x_5,\quad
z_{33}=x_6,\\[1ex]
\ds
a_{12}=\frac{x_3-x_7}{2},\quad
b_{12}=\frac{x_3+x_7}{2\sqrt{-1}},\quad
a_{13}=\frac{x_1-x_4}{2},\quad
b_{13}=\frac{x_1+x_4}{2\sqrt{-1}}
\end{array}
\]
Then, the polynomials $p_i(x)$ $(i=1,2,3)$ are transferred to
\[
p_1(z)=z_{11}z_{22}z_{33}-z_{22}(a_{13}^2+b_{13}^2)-z_{33}(a_{12}^2+b_{12}^2),\quad
p_2(z)=z_{22},\quad
p_3(z)=z_{33}.
\]
These three polynomials $p_i(z)$ $(i=1,2,3)$ are exactly all the basic relative invariants of a homogeneous open convex cones $\Omega$ defined by
\[
\Omega:=\set{Z=\pmat{z_{11}&\overline{z_{12}}&\overline{z_{13}}\\z_{12}&z_{22}&0\\z_{13}&0&z_{33}}\in\mathrm{Herm}(3,\C)}{\det Z>0,\,z_{22},z_{33}>0}.
\]
This cone $\Omega$ is a typical example of non-symmetric homogeneous open convex cone,
which can be viewed as a generalization of the so-called Vinberg cone.
Therefore, the corresponding prehomogeneous vector space associated with $p(x)$ is lieanrly isomorphic to
the prehomogeneous vector space obtained from the cone $\Omega$.
\end{Remark}

Note that the excluded case $n=2$ is calculated in Example \ref{exam:lie}.

\subsection{Circular cases}

Let $\mathtt{T}$ be a triangle arrangement which is constructed by arraying $n$ triangles circularly with $n\ge 5$
(see (3) in Figure~\ref{fig:general}).
In this case, the number of vertices are $2n$, that is, $\dim V=2n$.
Set $\varphi_n(k):=k\mathrm{\ mod\ }n\subset\{1,\dots,n\}$.
The $i$-th triangle consists of vertices $\{i,\,n+i,\,n+\varphi_n(i+1)\}$.
Then, the corresponding polynomial $p(x)$ is described as
\[
p(x)=\sum_{i=1}^nx_ix_{n+i}x_{n+\varphi_n(i+1)}.
\]
Let $E_{ij}$ be a matrix unit having $1$ on the position $(i,j)$ and zeros elsewhere.

\begin{Lemma}
A general element $M$ in $\mathfrak{g}[p]$ is of the form
\[
M=\pmat{
\mathrm{diag}(t_0-t_{i}-t_{\varphi_n(i+1)})_{i=1}^n&X\\
0&\mathrm{diag}(t_{i})_{i=1}^n
}
\]
where $X$ is an $n\times n$ matrix defined by 
\[
X=\sum_{i=1}^nX_{i+1,i}(E_{\varphi_n(i+1),i} - E_{i,\varphi_n(i+2)}).
\]
\end{Lemma}
\begin{proof}
Since the triangle arrangement $\mathtt{T}$ does not have edge sharing,
we can apply discussion in Section \ref{sect:structure}.
Let $j=1,\dots,n$.
Then, the vertex $j$ is isolated, and the equation \eqref{eq:key eq} can be written as
\[
M_{j,n+\varphi_n(j+2)}+M_{\varphi_n(j+1),n+j}=0\quad(j=1,\dots,n).
\]
The other terms are all zeros so that
the proof is now completed.
\end{proof}

This lemma implies that $\mathfrak{g}[p]$ is a solvable Lie algebra with $\dim \mathfrak{g}[p]=2n+1$.

\begin{Lemma}
Let $n\ge 5$.
The dual triplet $(\mathfrak{g}[p],d\rho^*,V)$ is also a prehomogeneous vector space
if and only if $n$ is an odd number.
If $n$ is odd,
then its basic relative invariants are given as
\[
q_0(x)=\sum_{i=1}^n x_{n+i}\prod_{j=0}^kx_{\varphi_n(i+2j)},\quad
q_i(x)=x_i\quad(i=1,\dots,n).
\]
Here,
indices of $x$ in the product symbol run through $1,\dots,n$ modulo $n$.
\end{Lemma}
\begin{proof}
We use Lemma~\ref{lemma}.
For simplicity,
we set $y_i:=x_{n+i}$ $(i=1,\dots,n)$
with indices being taken in $1,\dots,n$ by modulo $n$.
For $x\in V$,
let $A(x)\colon \mathfrak{g}[p]\to V$ be a linear map defined by $A(x)M:=d\rho(M)x$ $(M\in\mathfrak{g}[p])$.
Then, since $\dim\mathfrak{g}[p]=\dim V+1$, we see that $A(x)$ is a matrix of size $\dim V\times(\dim V+1)$.
By the above lemma, we have
{\tiny
\[
A(x)=
\left(
\begin{array}{c|ccccc|ccccc}
x_1&&&&&&-x_1&-x_1&0&\cdots&0\\
x_2&&&&&&0&-x_2&-x_2&\ddots&\vdots\\
\vdots&&&B&&&\vdots&\ddots&\ddots&\ddots&0\\
x_{n-1}&&&&&&0&&\ddots&-x_{n-1}&-x_{n-1}\\
x_n&&&&&&-x_n&0&\cdots&0&-x_n\\ \hline
0&0&\cdots&\cdots&\cdots&0&
y_1&0&0&\cdots&0\\
0&0&\cdots&\cdots&\cdots&0&
0&y_2&0&\ddots&\vdots\\
0&0&\cdots&\cdots&\cdots&0&
\vdots&\ddots&\ddots&\ddots&0\\
0&0&\cdots&\cdots&\cdots&0&
\vdots&&\ddots&y_{n-1}&0\\
0&0&\cdots&\cdots&\cdots&0&
0&0&\cdots&0&y_n
\end{array}
\right),
\]
}
where
\[B=\pmat{
y_n&-y_3&0&\cdots&\cdots&0\\
0&y_1&-y_4&\ddots&&\vdots\\
\vdots&0&y_2&\ddots&\ddots&\vdots\\
\vdots&\vdots&\ddots&\ddots&-y_n&0\\
0&0&\cdots&0&y_{n-2}&-y_1\\
-y_2&0&\cdots&\cdots&0&y_{n-1}
}.\]
Let $A'(x)$ be a square matrix obtained by removing the second column from $A(x)$.
Set
\[
B'=\pmat{
x_1&-y_3&0&\cdots&\cdots&0\\
x_2&y_1&-y_4&\ddots&&\vdots\\
\vdots&0&y_2&\ddots&\ddots&\vdots\\
\vdots&\vdots&\ddots&\ddots&-y_n&0\\
x_{n-1}&0&\cdots&0&y_{n-2}&-y_1\\
x_n&0&\cdots&\cdots&0&y_{n-1}
}.
\]
Then, we have 
\[\det A'(x)=
y_1\cdots y_n\times \det B'.
\]
By taking cofactor expansion on $B'$ at the first column,
the $i$-th row element is given as
\[
\begin{array}{l}
\ds
(-1)^{i+1}x_i\times \det
\underbrace{\pmat{-y_3&0&\cdots&0\\y_1&-y_4&\ddots&\vdots\\&\ddots&\ddots&0\\0&&y_{i-2}&-y_{i+1}}}_{i-1}
\times 
\det\underbrace{\pmat{y_{i}&-y_{i+3}&&0\\0&y_{i+1}&\ddots&\\ \vdots&\ddots&\ddots&-y_1\\0&\cdots&0&y_{n-1}}}_{n-i}\\
\ds
\quad
=(-1)^{i+1}x_i\times (-1)^{i-1}\bigl(y_3y_4\cdots y_{i+1}\bigr)
\times\bigl(y_{i+1}y_{i+2}\cdots y_{n-1}\bigr)\\
\ds
\quad
=
x_iy_iy_{i+1}\times y_3y_4\cdots y_{n-1}.
\end{array}
\]
Here, we ignore the first $\det$ when $i=1$, and ignore the last $\det$ when $i=n$,
and recall that indices run through $1,\dots,n$ modulo $n$.
This implies that
\[
\det A'(x)
=
y_1y_2y_3^2y_4^2\cdots y_{n-1}^2y_n\sum_{i=1}^nx_iy_iy_{i+1}
=
y_1y_2y_3^2y_4^2\cdots y_{n-1}^2y_n\times p(x),
\]
which shows that the triplet $(\mathfrak{g}[p],d\rho,V)$ is a prehomogeneous vector space,
and it is easily verified that polynomials $y_i=x_{n+i}$ $(i=1,\dots,n)$ are basic relative invariants.

Now let us investigate the dual prehomogeneous vector space.
Let us consider a linear map $A^*(x)\colon \mathfrak{h}\to V$ $(x\in V)$ defined by
$A^*(x)M:={}^{t\!}d\rho(M)x$ $(M\in\mathfrak{h})$.
Then, $A^*(x)$ is a matrix of size $2n\times(2n+1)$ given as
{\small
\[
A^*(x)=
\left(
\begin{array}{c|ccccc|cccccc}
x_1&-x_1&-x_1&&&&&&	\\
x_2&&-x_2&-x_2&&&&&&	\\
\vdots&&&\ddots&\ddots&&&&	\\
x_{n-1}&&&&-x_{n-1}&-x_{n-1}&&\\
x_n&-x_n&&&&-x_n&	\\ \hline
0&y_1&&&&&	x_2&0&&-x_{n-1}&0\\
0&&y_2&&&&	0&x_3&\ddots&&-x_n\\
\vdots&&&\ddots&&&	-x_1&\ddots&\ddots&0\\
\vdots&&&&\ddots&&	0&\ddots&\ddots&x_n&0\\
0&&&&&y_{n}&x_n	&0&-x_{n-2}&0&x_1\\
\end{array}
\right).
\]
}
Let us consider a rank of the following matrix.
\[
B(x):=
\pmat{x_1&-x_1&-x_1	\\
x_2&&-x_2&-x_2	\\
\vdots&&&\ddots&\ddots	\\
x_{n-1}&&&&-x_{n-1}&-x_{n-1}\\
x_n&-x_n&&&&-x_n}.
\]
If $n$ is even, then sum of columns $2,4,\dots,n$ is equal to ${}^{t\!}(-x_1,\dots,-x_n)$, and so is that of columns $3,5,\dots,n+1$.
This means that the rank of $B(x)$ must be smaller than or equal to $n-1$,
which implies $\mathrm{rank}\, A^*(x)\le 2n-1$.
Since the size of $A^*(x)$ is $2n$,
$A^*(x)$ cannot be full rank.
Thus, if $n\ge 5$ is even, then the dual triplet $(\mathfrak{g}[p],d\rho^*,V)$ cannot be a prehomogneous vector space.

Next, we assume that $n$ is odd.
We shall calculate a determinant of a matrix
obtained by removing the last column from $A^*(x)$.
\[
A'(x)=
\left(
\begin{array}{c|ccccc|ccccc}
x_1&-x_1&-x_1&&&&&&	\\
x_2&&-x_2&-x_2&&&&&&	\\
\vdots&&&\ddots&\ddots&&&&	\\
x_{n-1}&&&&-x_{n-1}&-x_{n-1}&&\\
x_n&-x_n&&&&-x_n&	\\ \hline
0&y_1&&&&&	x_2&0&&-x_{n-1}\\
0&&y_2&&&&	0&x_3&\ddots&\\
\vdots&&&\ddots&&&	-x_1&\ddots&\ddots&0\\
\vdots&&&&\ddots&&	0&\ddots&\ddots&x_n\\
0&&&&&y_{n}&x_n	&0&-x_{n-2}&0\\
\end{array}
\right).
\]
At first,
we add the columns $2,3,\dots,n$ and twice of the column $1$ 
to the column $n+1$
to obtain
{\small
\[
\det A'(x)=
\left(
\begin{array}{c|ccccc|ccccc}
x_1&-x_1&-x_1&&&0&&&	\\
x_2&&-x_2&\ddots&&0&&&&	\\
\vdots&&&\ddots&-x_{n-2}&\vdots&&&	\\
x_{n-1}&&&&-x_{n-1}&0&&\\
x_n&-x_n&&&0&0&	\\ \hline
0&y_1&&&&y_1&	x_2&0&&-x_{n-1}\\
0&&y_2&&&y_2&	0&x_3&\ddots&\\
\vdots&&&\ddots&&\vdots&	-x_1&\ddots&\ddots&0\\
\vdots&&&&\ddots&y_{n-1}&	0&\ddots&\ddots&x_n\\
0&&&&&y_{n}&x_n	&0&-x_{n-2}&0\\
\end{array}
\right).
\]
}
Thus, wee see that $\det A'(x)$ can be factorized as a product of
\[
A_1=
\det \pmat{x_1&-x_1&-x_1&&&\\
x_2&&-x_2&\ddots&&\\
\vdots&&&\ddots&-x_{n-2}	\\
x_{n-1}&&&&-x_{n-1}\\
x_n&-x_n&&&0}
\]
and
\[
A_2=
\det\pmat{
y_1&	x_2&0&&-x_{n-1}\\
y_2&	0&x_3&\ddots&\\
\vdots&	-x_1&\ddots&\ddots&0\\
y_{n-1}&	0&\ddots&\ddots&x_n\\
y_{n}&x_n	&0&-x_{n-2}&0
}.
\]
Let us calculate $A_1$ and $A_2$.
Set
\[
B_n=\det\pmat{1&-1&&&0\\
1&-1&-1	\\
\vdots&&\ddots&\ddots	\\
1&&&-1&-1\\
1&0&&&-1}.\]
Then, $A_1$ can be expressed, up to signature, as
\[
A_1=x_1\cdots x_n\cdot 
\det \pmat{1&-1&-1&&&\\
1&&-1&\ddots&&\\
\vdots&&&\ddots&-1\\
1&&&&-1\\
1&-1&&&0}
=
(\mathrm{sgn})x_1\cdots x_n\cdot B_n.
\]
By a cofactor expansion along the last row,
we see that $B_n$ can be calculated as
\[
\begin{array}{r@{\ }c@{\ }l}
B_n
&=&
(-1)^{n+1}\cdot(-1)^{n-1}+(-1)\cdot B_{n-1}
=
1-B_{n-1}\\
&=&
1-((-1)^{n-1+1}\cdot(-1)^{n-2}+(-1)\cdot B_{n-2})\\
&=&
B_{n-2}
=
\cdots
=
B_3=1,
\end{array}
\]
whence
\[
A_1=(\mathrm{sgn})\,x_1\cdots x_n.
\]
Next we consider $A_2$.
Let $n=2k+1$.
By changing an order of columns as $1$, $2,4,\dots,2k$ and then $3,5,\dots,2k+1$
and that of rows as $1,3,\dots,2k+1$ and then $2,4,\dots,2k$ 
to obtain
\[
A_2=\left|\begin{array}{c|cccc|cccc}
y_1&x_2&&&&&&&-x_{2k}\\
y_3&-x_1&x_4&&&&\\
\vdots&&\ddots&\ddots&&&\\
y_{2k-1}&&&-x_{2k-3}&x_{2k}\\
y_{2k+1}&&&&-x_{2k-1}\\ \hline
y_2&&&&&x_3\\
y_4&&&&&-x_2&x_5\\
\vdots&&&&&&\ddots&\ddots\\
y_{2k}&&&&&&&-x_{2k-2}&x_{2k+1}
\end{array}\right|.
\]
By using a cofactor expansion along the fist column,
we obtain
\[
\begin{array}{l}
\ds
A_2=\sum_{i=1}^{k+1}
(-1)^{i+1}y_{2i-1}\prod_{j=i}^{k}(-x_{2j-1})\cdot \prod_{j=1}^{i-1}x_{2j}\cdot
\prod_{j=1}^kx_{2j+1}\\
\ds
\qquad
+
\sum_{i=1}^k
(-1)^{k+i}y_{2i}\cdot (-1)^k \,x_{2k}\prod_{j=1}^{k}x_{2j-1}
\cdot
\prod_{j=i}^{k-1}(-x_{2j})
\prod_{j=1}^{i-1}x_{2j+1}\\
\ds
\quad
=
(-1)^kx_3x_5\cdots x_{2k-1}q(x,y),
\end{array}
\]
where $q$ is a homogeneous polynomial of degree $k+2$ defined by
\[
q(x,y):=\sum_{i=1}^{k+1}
y_{2i-1}x_{2i-1}x_{2i+1}\cdots x_{2k+1}\cdot x_2x_4\cdots x_{2i-2}
+
\sum_{i=1}^k
y_{2i}x_{2i}x_{2i+2}\cdots x_{2k}\cdot x_{1}x_3\cdots x_{2i-1}.
\]
Since $y_i=x_{n+i}$,
we see that
\[
q(x,y)=\sum_{i=1}^n x_{n+i}\prod_{j=0}^kx_{\varphi_n(i+2j)}=q_0(x)
\]
and thus a generic rank of $A^*(x)$ is full so that
the dual triplet $(\mathfrak{g}[p],d\rho^*,V)$ is a prehomogeneous vector space.
By structure of $\mathfrak{g}[p]$,
it is easily verified that polynomials $q_i(x)=x_i$ $(i=1,\dots,n)$ are relatively invariant under the action of $\mathfrak{g}[p]$.
The proof is now completed.
\end{proof}

\begin{Remark}
Let $n=2k+1$ be an odd number.
Then, $q_0(x)$ is a homogeneous polynomial of degree $k+2=\frac{n+3}{2}$.
This is an example that, for a given arbitrary integer $N$,
we can construct a relative invariant of degree $N$ of a prehomogeneous vector space.
\end{Remark}

In what follows, we deal with cases $n=3,4$.
Both cases correspond to prehomogeneous vector spaces,
and general elements $M$ of $\mathfrak{g}[p]$ are given 
for the case $n=3$ as
\[
M=\pmat{
t_0-t_1&0&0&X_{1}&0&0\\
0&t_0-t_2&0&0&X_{2}-X_1&0\\
0&0&t_0-t_3&0&0&-X_2\\
0&0&0&t_1&0&0\\
0&0&0&0&t_2&0\\
0&0&0&0&0&t_3
},
\]
and for the case $n=4$ as
{\tiny
\[
M=
\left(
\begin{array}{*{4}{c}|*{4}{c}}
t_0-t_1-t_2&Y_{12}&0&Y_{23}&	0&0&-X_{21}&X_{14}\\
Y_{12}&t_0-t_2-t_3&Y_{23}&0&	X_{21}&0&0&-X_{32}\\
0&Y_{32}&t_0-t_3-t_4&Y_{21}&	-X_{43}&X_{32}&0&0\\
Y_{32}&0&Y_{12}&t_0-t_4-t_1&	0&-X_{14}&X_{43}&0\\ \hline
0&0&0&0&	t_1&0&-Y_{21}&0\\
0&0&0&0&	0&t_2&0&-Y_{32}\\
0&0&0&0&	-Y_{12}&0&t_3&0\\
0&0&0&0&	0&-Y_{23}&0&t_4
\end{array}
\right).
\]
}
Note that both cases are excluded from detailed discussion in Section~\ref{sect:structure}.
In particular, the matrices of case $n=4$ has the same $(1,2)$ block of the case $n\ge 5$,
but it has additional variables $Y_{ij}$ in diagonal blocks.
Moreover, for the case $n=4$,
the prehomogeneous vector space $(\mathfrak{g}[p],d\rho,V)$ is not regular,
but its dual $(\mathfrak{g}[p],d\rho^*,V)$ is also a prehomogeneous vector space with a unique basic relative invariant
\[
q_0(x)=x_1x_3-x_2x_4\quad(x\in V).
\]
We are also able to confirm that $(\mathfrak{g}[p],d\rho,V)$ is not regular by checking that 
not all variables appear in $q_0(x)$.

\subsection{Edge gluing cases}

Let $\mathtt{T}$ be a triangle arrangement constructed by gluing $n$ triangle arrangements $\mathtt{T}_B$ at edges $24$ and $36$.
Note that $\dim V=4n+2$.
To each vertex, we attach variables $x_i,y_i$ $(i=0,1,\dots,n)$ and $z_j,w_j$ $(j=1,\dots,n)$ as in Figure~\ref{fig:mitsu}.
Then, the corresponding polynomial $p(x)$ is described as
\[
p(\bs{x})=\sum_{i=1}^n\bigl(x_ix_{i+1}z_i+x_iy_{i+1}w_i+x_{i+1}y_{i+1}w_i\bigr)\quad
(\bs{x}=(x,y,z,w)\in V).
\]
The order of basis of $V$ is $x$, $w$, $y$ and $z$.
Let $E_{ij}^{(m)}$ (resp.\ $E'_{ij}$) be a matrix unit of size $m\times m$ (resp.\ $n\times (n+1)$) having one on the position $(i,j)$ and zeros elsewhere.

\begin{figure}
    \centering
    \includegraphics[scale=0.4]{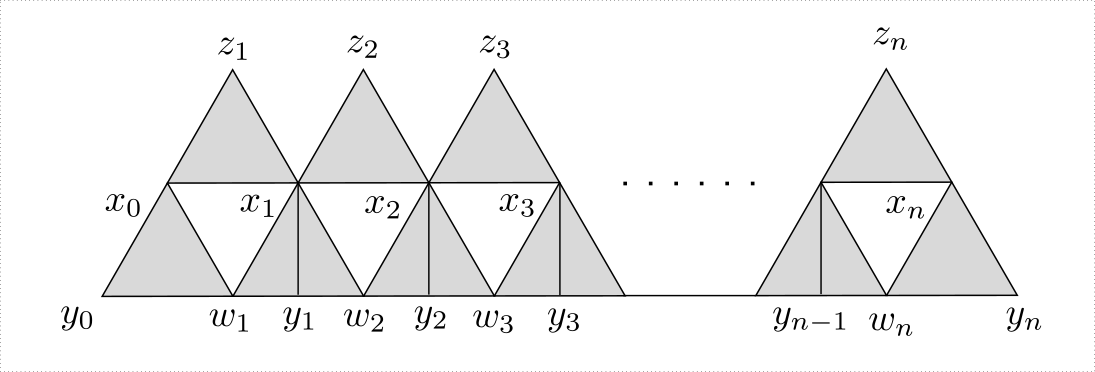}
    \caption{Labelling for edge gluing cases}
    \label{fig:mitsu}
\end{figure}

\begin{Lemma}
\label{lemma 5 10}
A general element $M$ in $\mathfrak{g}[p]$ is of the form
\[
M=\pmat{
M^{xx}&0&0&0\\
M^{wx}&M^{ww}&0&0\\
M^{yx}&0&M^{yy}&0\\
M^{zx}&M^{zw}&M^{zy}&M^{zz}
}
\]
where $M^{ab}$ $(a,b=x,y,z,w)$ are given as 
\[
\begin{cases}
M^{xx}=\mathrm{diag}(t_i)_{i=1}^{n+1},\\
M^{ww}=uI_n,\\
M^{yy}=\mathrm{diag}(t-u-t_i)_{i=1}^{n+1},\\
M^{zz}=\mathrm{diag}(t-t_i-t_{i+1})_{i=1}^n\\
\ds M^{wx}=\sum_{i=1}^{n-1}d_i(E'_{i,i+1}-E'_{i+1,i+1}),\\
\ds M^{zy}=\sum_{i=1}^{n-1}d_i(E'_{i+1,i+2}-E'_{ii})\\
\ds M^{yx}=\sum_{i=1}^nb_iE_{i,i+1}^{(n+1)}+c_iE_{i+1,i}^{(n+1)},\\
\ds M^{zw}=
-\sum_{i=1}^n
b_iE_{i,i-1}^{(n)}+(b_i+c_i)E_{ii}^{(n)}+c_iE_{i,i+1}^{(n)}\\
M^{zx}=\sum_{i=1}^{n-1}a_i(E'_{i,i+2}-E'_{i+1,i}).
\end{cases}
\]
\end{Lemma}

We shall give a proof of this lemma at the end of this sebsection.
This lemma shows that
$\mathfrak{g}[p]$ is a solvable Lie algebra and
$\dim \mathfrak{g}[p]=5n+1$.

\begin{Lemma}
\label{lemma 5 11}
Let $n\ge 2$.
Then, the triplet $(\mathfrak{g}[p],d\rho,V)$ is a regular prehomogeneous vector space.
Its basic relative invariants are described as
\[
p_i(\bs{x})=x_i\quad(i=0,1,\dots,n),\quad
p_{n+1}(\bs{x})=w_1+\cdots+w_n,\quad
p_{n+2}(\bs{x})=p(x).
\]
\end{Lemma}

\begin{proof}[Proof of Lemma~\ref{lemma 5 10}]
We shall calculate directly by using \eqref{eq:InvLie}.
The gradients of $p(x)$ can be calculated as

\[
\nabla_xp(\bs{x})=(\overset{i=0}{\overbrace{x_1z_1+y_0w_1}},\ \overset{i}{\overbrace{x_{i-1}z_i+y_iw_i+x_{i+1}z_{i+1}+y_iw_{i+1}}},\ 
\overset{i=n}{\overbrace{w_ny_n+x_{n-1}z_n}})
\]

\[
\nabla_yp(\bs{x})=(\overset{i=0}{\overbrace{x_0w_1}},\ \overset{i}{\overbrace{x_iw_i+x_iw_{i+1}}},\ \overset{i=n}{\overbrace{x_nw_n}})
\]

\[
\nabla_zp(\bs{x})=(\overset{i}{\overbrace{x_{i-1}x_i}})
\]

\[
\nabla_wp(\bs{x})=(\overset{i}{\overbrace{x_{i-1}y_{i-1}+x_iy_i}})
\]

Thus, $\innV{M\bs{x}}{\nabla_{\bs{x}}p(\bs{x})}$ can be expanded as
\[
\begin{array}{r@{\ }c@{\ }l}
\innV{M\bs{x}}{\nabla_{\bs{x}}p(\bs{x})}
&=&
\ds
\sum M^{zb}_{ij}b_jx_{i-1}x_i\\
& &
\ds
+\sum M^{wb}_{ij}b_j(x_{i-1}y_{i-1}+x_iy_i)\\
& &
\ds
+\sum M^{xb}_{sj}b_j(x_{s-1}z_s+y_sw_s+x_{s+1}z_{s+1}+y_sw_{s+1})\\
& &
\ds
+\sum M^{yb}_{sj}b_j(x_sw_s+x_sw_{s+1}).
\end{array}
\]
Here,
the summation symbol $\sum$ runs over all possible $i,j$ and $b=x,y,z,w$.
Among $M^{ab}_{ij}$ $(a,b=x,y,z,w)$,
candidates such that $M^{ab}_{ij}\ne 0$ are as follows.
\[
M^{xx}_{ss},
\]
\[
M^{yz}_{ii},\ M^{yz}_{i,i+1},\ M^{yy}_{ss},\ M^{yx}_{i,i-1},\ M^{yx}_{i-1,i}
\]
\[
M^{zx}_{j,j+1},\ M^{zx}_{j+1,j-1},\ M^{zy}_{ii},\ M^{zy}_{i,i-1},\ M^{zz}_{ii},\ M^{zz}_{j,j+1},\ M^{zz}_{j+1,j},\ M^{zw}_{ii},\ M^{zw}_{j,j+1},\ M^{zw}_{j+1,j}
\]
\[M^{wx}_{i,i-1},\ M^{wx}_{ii},\ M^{wz}_{ii},\ M^{ww}_{ii}\]
Here, we set
\[
s=0,1,\dots,n,\ i=1,\dots,n,\quad j=1,\dots,n-1.
\]

We shall calculate for each monomial.
\begin{enumerate}
\item $(M^{zz}_{i,i-1}+M^{xx}_{i-2,i})x_{i-1}x_iz_{i-1}$ $(i=2,\dots,n)$

Since $M^{xx}_{i-1,i}=0$m we have $M^{zz}_{i,i-1}=0$.

\item $(M^{zz}_{i,i+1}+M^{xx}_{i+1,i-1})x_{i-1}x_iz_{i+1}$ $(i=1,\dots,n-1)$

Since $M^{xx}_{i+1,i-1}=0$, we have $M^{zz}_{i,i+1}=0$.

\item $(M^{zx}_{i,i-2}+M^{zx}_{i-1,i})x_{i-2}x_{i-1}x_{i}$ $(i=2,\dots,n)$

We have $M^{zx}_{i,i-2}+M^{zx}_{i-1,i}=0$.

\item $(M^{zw}_{i,i-1}+M^{yx}_{i-1,i})x_{i-1}x_iw_{i-1}$ $(i=2,\dots,n)$

We have $M^{zw}_{i,i-1}+M^{yx}_{i-1,i}=0$

\item $(M^{zw}_{ii}+M^{yx}_{i-1,i}+M^{yx}_{i,i-1})x_{i-1}x_iw_i$ $(i=1,\dots,n)$

We have $M^{zw}_{ii}+M^{yx}_{i-1,i}+M^{yx}_{i,i-1}=0$

\item $(M^{zw}_{i,i+1}+M^{yx}_{i,i-1})x_{i-1}x_iw_{i+1}$ $(i=1,\dots,n-1)$

We have $M^{zw}_{i,i+1}+M^{yx}_{i,i-1}=0$

\item $(M^{zy}_{i,i-1}+M^{wx}_{i-1,i}+M^{wx}_{ii})x_{i-1}x_iy_{i-1}$ $(i=1,\dots,n)$

Since $M^{wx}_{i-1,i}=0$, we have $M^{zy}_{i,i-1}+M^{wx}_{ii}=0$

\item $(M^{zy}_{ii}+M^{wx}_{i,i-1}+M^{wx}_{i+1,i-1})x_{i-1}x_iy_i$ $(i=1,\dots,n)$

Since $M^{wx}_{i+1,i-1}=0$, we have $M^{zy}_{ii}+M^{wx}_{i,i-1}=0$

\item $(M^{xz}_{i,i-1}+M^{xz}_{i-1,i})x_{i-1}z_{i-1}z_i$ $(i=2,\dots,n)$

We have $M^{xz}_{i,i-1}=M^{xz}_{i-1,i}=0$

\item $(M^{xw}_{i,i-1}+M^{yz}_{i-1,i})x_{i-1}z_iw_{i-1}$ $(i=2,\dots,n)$

Since $M^{xw}_{i,i-1}=0$, we have $M^{yz}_{i-1,i}=0$

\item $(M^{xw}_{ii}+M^{yz}_{i-1,i})x_{i-1}z_iw_i$ $(i=1,\dots,n)$

Since $M^{xw}_{ii}=0$, we have $M^{yz}_{i-1,i}=0$

\item $(M^{xy}_{i,i-1}+M^{wz}_{i-1,i}+M^{wz}_{ii})x_{i-1}y_{i-1}z_i$ $(i=1,\dots,n)$

Since $M^{wz}_{i-1,i}=0$, we have $M^{xy}_{i,i-1}+M^{wz}_{ii}=0$

\item $(M^{xw}_{i-1,i}+M^{yz}_{ii})x_iz_iw_i$ $(i=1,\dots,n)$

Since $M^{xw}_{i-1,i}=0$, we have $M^{yz}_{ii}=0$

\item $(M^{xw}_{i-1,i+1}+M^{yz}_{ii})x_iz_iw_{i+1}$ $(i=1,\dots,n-1)$

Since $M^{xw}_{i-1,i+1}=0$, we have $M^{yz}_{ii}=0$

\item $(M^{xy}_{i-1,i}+M^{wz}_{ii}+M^{wx}_{i+1,i})x_iy_iz_i$ $(i=1,\dots,n)$

Since $M^{wx}_{i+1,i}=0$, we have $M^{xy}_{i-1,i}+M^{wz}_{ii}=0$

\item $(M^{yw}_{i-1,i-1}+M^{yw}_{i-1,i})x_{i-1}w_{i-1}w_i$ $(i=2,\dots,n)$

We have $M^{yw}_{i-1,i-1}=M^{yw}_{i-1,i}=0$

\item $(M^{yy}_{i-1,i}+M^{xx}_{i,i-1})x_{i-1}y_iw_i$ $(i=1,\dots,n)$

We have $M^{yy}_{i-1,i}=M^{xx}_{i,i-1}=0$

\item $(M^{xw}_{i-1,i-1}+M^{xw}_{i-1,i})y_{i-1}w_{i-1}w_i$ $(i=2,\dots,n)$

We have $M^{xw}_{i-1,i-1}=M^{xw}_{i-1,i}=0$

\item $(M^{xy}_{i-1,i}+M^{xy}_{i,i-1})y_{i-1}y_iw_i$ $(i=1,\dots,n)$

We have $M^{xy}_{01}=0$, $M^{xy}_{n,n-1}=0$ and
$M^{xy}_{i-1,i}+M^{xy}_{i,i-1}=0$

\item $(M^{wx}_{i-1,i-1}+M^{wx}_{i,i-1})x_{i-1}^2y_{i-1}$ $(i=1,\dots,n+1)$

We have $M^{wx}_{10}=0$, $M^{wx}_{nn}=0$ and $M^{wx}_{i-1,i-1}+M^{wx}_{i,i-1}=0$

\item[] Diagonal entries

\item $(M^{xx}_{i-1,i-1}+M^{xx}_{ii}+M^{zz}_{ii})x_{i-1}x_iz_i$ $(i=1,\dots,n)$

We have $M^{xx}_{i-1,i-1}+M^{xx}_{ii}+M^{zz}_{ii}=t$

\item $(M^{yy}_{i-1,i-1}+M^{xx}_{i-1,i-1}+M^{ww}_{i-1,i}+M^{ww}_{ii})x_{i-1}y_{i-1}w_i$ $(i=1,\dots,n)$

Since $M^{ww}_{i-1,i}=0$, we have
$M^{yy}_{i-1,i-1}+M^{xx}_{i-1,i-1}+M^{ww}_{ii}=t$

\item $(M^{yy}_{ii}+M^{xx}_{ii}+M^{ww}_{ii}+M^{ww}_{i+1,i})x_iy_iw_i$ $(i=1,\dots,n)$

Since
$M^{ww}_{i+1,i}=0$, we have
$M^{yy}_{ii}+M^{xx}_{ii}+M^{ww}_{ii}=t$
\end{enumerate}
Summing up the above discussion,
we see that 3 implies
\[
M^{zx}_{i,i-2}+M^{zx}_{i-1,i}=0\quad(i=2,\dots,n)
\]
($\dim = n-1$),
4,5,6 imply
\[
\begin{cases}
M^{zw}_{i,i-1}+M^{yx}_{i-1,i}=0&(i=2,\dots,n)\\
M^{zw}_{ii}+M^{yx}_{i-1,i}+M^{yx}_{i,i-1}=0&(i=1,\dots,n)\\
M^{zw}_{i,i+1}+M^{yx}_{i,i-1}=0&(i=1,\dots,n-1)
\end{cases}
\]
($\dim =2n$), 
7,8,20 tell us that
\[
\begin{cases}
M^{zy}_{i,i-1}+M^{wx}_{ii}=0&(i=1,\dots,n)\\
M^{zy}_{ii}+M^{wx}_{i,i-1}=0&(i=1,\dots,n)\\
M^{wx}_{i-1,i-1}+M^{wx}_{i,i-1}=0&(i=2,\dots,n-1)\\
M^{wx}_{10}=0\\
M^{wx}_{nn}=0
\end{cases}
\]
($\dim =n-1$).
Moreover, 12,15,19 yield that
\[
\begin{cases}
M^{xy}_{i,i-1}+M^{wz}_{ii}=0&(i=1,\dots,n)\\
M^{xy}_{i-1,i}+M^{wz}_{ii}=0&(i=1,\dots,n)\\
M^{xy}_{i-1,i}+M^{xy}_{i,i-1}=0&(i=2,\dots,n-1)\\
M^{xy}_{01}=0\\
M^{xy}_{n,n-1}=0
\end{cases}
\]
and hence one obtain
\[
M^{wz}_{ii}=0,\quad\text{whence}\quad
M^{xy}_{i,i-1}=M^{xy}_{i-1,i}=0
\]
($\dim = 0$). 
Diagonal entries are 
\[
\begin{cases}
M^{xx}_{i-1,i-1}+M^{xx}_{ii}+M^{zz}_{ii}=t&(i=1,\dots,n)\\
M^{yy}_{i-1,i-1}+M^{xx}_{i-1,i-1}+M^{ww}_{ii}=t&(i=1,\dots,n)\\
M^{yy}_{ii}+M^{xx}_{ii}+M^{ww}_{ii}=t&(i=1,\dots,n)
\end{cases}
\]
($\dim = n+3$). 
The second and  third lines shows that 
$M^{ww}_{ii}$ does not depend on $i$.
Therefore,
let us introduce new variables 
\[
a_{i}=M^{zx}_{i,i+1},\quad
d_i=M^{wx}_{ii}\quad(i=1,\dots,n-1),\quad
b_{j}=M^{yx}_{j,j+1},\quad
c_j=M^{yx}_{j+1,j}\quad(j=0,1,\dots,n-1)
\]
and
\[
t,\quad
M^x_i=M^{xx}_{i}\quad(i=0,1,\dots,n),\quad
M^y=g^{yy}_{00}
\]
as a basis.
Then we have
\[
M^{ww}_{ii}=t-M^x_0-y\quad(i=1,\dots,n),\quad
M^{yy}_{jj}=t-M^x_j-(t-g^x_0-y)
=
M^x_0-M^x_j+y\quad(j=0,1,\dots,n).
\]
An element $M\in\mathfrak{g}[p]$ can be described as a block matrix form of
\[
M=\pmat{
M^{xx}&0&0&0\\
M^{wx}&M^{ww}&0&0\\
M^{yx}&0&M^{yy}&0\\
M^{zx}&M^{zw}&M^{zy}&M^{zz}
},
\]
where

\[
\begin{array}{r@{\ }c@{\ }l}
M^{yx}
&=&
\ds
\pmat{
0&M^{yx}_{01}&0&\cdots&0\\
M^{yx}_{10}&0&M^{yx}_{12}&\ddots&\vdots\\
0&\ddots&\ddots&\ddots&0\\
\vdots&\ddots&M^{yx}_{n-1,n-2}&0&M^{yx}_{n-1,n}\\
0&\cdots&0&M^{yx}_{n,n-1}&0
}
\in\mathrm{Mat}(n+1;\,\C)\\
&=&
\ds
\pmat{
0&b_0&0&\cdots&0\\
c_0&0&b_1&\ddots&\vdots\\
0&\ddots&\ddots&\ddots&0\\
\vdots&\ddots&c_{n-2}&0&b_{n-1}\\
0&\cdots&0&c_{n-1}&0
}
\end{array}
\]
\[
\begin{array}{r@{\ }c@{\ }l}
M^{zx}
&=&
\ds
\pmat{
0&0&M^{zx}_{12}&0&\cdots&0\\
M^{zx}_{20}&0&0&\ddots&\ddots&\vdots\\
0&M^{zx}_{31}&\ddots&\ddots&M^{zx}_{n-2,n-3}&0\\
\vdots&\ddots&\ddots&&0&M^{zx}_{n-1,n-2}\\
0&\cdots&0&M^{zx}_{n,n-2}&0&0
}\in\mathrm{Mat}(n\times (n+1);\,\C)\\
&=&
\ds
\pmat{
0&0&a_1&0&\cdots&0\\
-a_1&0&0&\ddots&\ddots&\vdots\\
0&-a_2&\ddots&\ddots&a_{n-2}&0\\
\vdots&\ddots&\ddots&&0&a_{n-1}\\
0&\cdots&0&-a_{n-1}&0&0
}
\end{array}
\]
\[
\begin{array}{r@{\ }c@{\ }l}
M^{wx}
&=&
\ds
\pmat{
M^{wx}_{10}&M^{wx}_{11}&0&\cdots&0\\
0&M^{wx}_{21}&M^{wx}_{22}&\ddots&\vdots\\
\vdots&\ddots&\ddots&\ddots&0\\
0&\cdots&0&M^{wx}_{n,n-1}&M^{wx}_{nn}
}
\in\mathrm{Mat}(n\times(n+1);\,\C)\\
&=&
\ds
\pmat{
0&d_1&0&\cdots&0&0\\
0&-d_1&d_{2}&\ddots&\vdots&\vdots\\
\vdots&\ddots&\ddots&\ddots&0&0\\
&&&-d_{n-2}&d_{n-1}&0\\
0&\cdots&\cdots&0&-d_{n-1}&0
}
\end{array}
\]
\[
\begin{array}{r@{\ }c@{\ }l}
M^{zy}
&=&
\ds
\pmat{
M^{zy}_{10}&M^{zy}_{11}&0&\cdots&0\\
0&M^{zy}_{21}&M^{zy}_{22}&\ddots&\vdots\\
\vdots&\ddots&\ddots&\ddots&0\\
0&\cdots&0&M^{zy}_{n,n-1}&M^{zy}_{nn}
}
\in\mathrm{Mat}(n\times(n+1);\,\C)\\
&=&
\ds
\pmat{
-d_{1}&0&0&\cdots&0&0\\
0&-d_{2}&d_{1}&\ddots&\vdots&\vdots\\
\vdots&\ddots&\ddots&\ddots&0&0\\
0&\cdots&0&-d_{n-1}&d_{n-2}&0\\
0&\cdots&0&0&0&d_{n-1}
}
\end{array}
\]
\[
\begin{array}{r@{\ }c@{\ }l}
M^{zw}
&=&
\ds
\pmat{
M^{zw}_{11}&M^{zw}_{12}&0&\cdots&0\\
M^{zw}_{21}&M^{zw}_{22}&M^{zw}_{23}&\ddots&\vdots\\
0&\ddots&\ddots&\ddots&0\\
\vdots&\ddots&M^{zw}_{n-1,n-2}&M^{zw}_{n-1,n-1}&M^{zw}_{n-1,n}\\
0&\cdots&0&M^{zw}_{n,n-1}&M^{zw}_{nn}
}
\in\mathrm{Mat}(n\times n;\,\C)\\
&=&
\ds
\pmat{
-b_0-c_0&-c_0&0&\cdots&0\\
-b_1&-b_1-c_1&-c_1&\ddots&\vdots\\
0&\ddots&\ddots&\ddots&0\\
\vdots&\ddots&-b_{n-2}&-b_{n-2}-c_{n-2}&-c_{n-2}\\
0&\cdots&0&-b_{n-1}&-b_{n-1}-c_{n-1}
}
\end{array}
\]
This proves the assertion.
\end{proof}

\begin{proof}[Proof of Lemma~\ref{lemma 5 11}]
Let $A(\bs{x})M:=d\rho(M)\bs{x}$.
We set an order of $\mathfrak{g}[p]$ by
$M^x_0,M^x_1,\dots,M^x_n,M^y,t$, and then $d_i$, $b_i$, $c_i$, $a_i$.
Recall that the order of $V$ is taken as $x,w,y,z$.
Then, $A(\bs{x})$ can be described as a block matrix form as
\[
A(\bs{x})=
\pmat{
D_1(x)&0&0&0&0\\
D_2(w)&X_1&0&0&0\\
D_3(y)&0&X_2&X_3&0\\
D_4(z)&Y_1&W_1&W_2&X_4
},
\]
where
\[
D_1(x)=\pmat{x_0&&0&0&0\\ &\ddots&&0&0\\0&&x_n&0&0}
\]
\[
D_2(w)=\pmat{
-w_1&0\cdots&0&-w_1&w_1\\
\vdots&\vdots\cdots&\vdots&\vdots&\vdots\\
-w_n&0\cdots&0&-w_n&w_n\\
}
\]
\[
D_3(y)=
\pmat{
0&0&\cdots&0&y_0&0\\
y_1&-y_1&&0&y_1&0\\
\vdots&&\ddots&\vdots&\vdots\\
y_n&0&&-y_n&y_n&0
}
\]
\[
D_4(z)=
\pmat{
-z_1&-z_1&&&0&0&z_1\\
&-z_2&-z_2&&&0&z_2\\
&&\ddots&\ddots&&\vdots&\vdots\\
0&&&-z_n&-z_n&0&z_n
}
\]
\[
X_1=\pmat{
x_1\\
-x_1&x_2\\
&\ddots&\ddots\\
&&-x_{n-2}&x_{n-1}\\
&&&-x_{n-1}
}
\]
\[
(X_2X_3)=
\left(\begin{array}{ccc|ccccc}
x_1&&&
	0&\cdots&0\\
&\ddots&&
	x_0\\
&&x_n&
	&\ddots\\
0&\cdots&0&
	&&x_{n-1}
\end{array}\right)
\]
\[
X_4=
\pmat{
x_2\\
-x_0&x_3\\
&\ddots&\ddots\\
&&-x_{n-3}&x_{n}\\
&&&-x_{n-2}
}
\]
\[
Y_1=\pmat{
-y_0\\
y_2&-y_1\\
&\ddots&\ddots\\
&&y_{n-1}&-y_{n-2}\\
&&&y_n
}
\]
\[
(W_1W_2)=
\left(\begin{array}{cccc|ccccc}
-w_1&&&&
	-w_1-w_2\\
&-w_1-w_2&&
	&&\ddots\\
&&\ddots&
	&&&-w_{n-1}-w_n\\
&&&-w_{n-1}-w_n&
	&&&-w_n
\end{array}\right)
\]
Removing columns corresponding to $c_0,c_1,\dots,c_{n-2}$ from $A(\bs{x})$,
we obtain a square matrix $B(\bs{x})$ of size $4n+2$ as
\[
B(\bs{x})
=   
\begin{array}{cc}
\begin{array}{cccccc}
n+1&\quad1&\quad1&n-1&n+1&n-1
\end{array}&\\
\left(
\begin{array}{cccccc}
D'_1(x)&0&0&0&0&0\\
D'_2(w)&-\bs{w}&\bs{w}&X_1&0&0\\
D'_3(y)&0&\bs{y}&0&D_5(x)&0\\
D'_4(z)&0&\bs{z}&Y_1&W'&X_2
\end{array}
\right)&\begin{array}{l}
n+1\\ n\\ n\\ n+1
\end{array}
\end{array}
\]
where
$D'_i$ are matrices obtained by removing the last two columns from $D_i$,
and $D_5(x)$, $W'$ are matrices defined by
\[
D_5(x)=\mathrm{diag}(x_1,x_2,\dots,x_n,x_{n-1})
\]
\[
W'=\pmat{
-w_1&&&&\\
&-w_1-w_2&&\\
&&\ddots&\\
&&&-w_{n-1}-w_n&-w_n}
\]
Let us calculate $\det B(\bs{x})$.
\[
\begin{array}{r@{\ }c@{\ }l}
\det B(\bs{x})
&=&
\ds
\det D'_1(x)
\det\pmat{-\bs{w}&\bs{w}&X_1&0&0\\
0&\bs{y}&0&D_5(x)&0\\
0&\bs{z}&Y_1&W'&X_4}\\
&=&
\ds
(\mathrm{sgn})
\det D'_1(x)
\det\pmat{
\bs{w}&X_1&0&0&0\\
0&0&\bs{y}&D_5(x)&0\\
0&Y_1&\bs{z}&W'&X_4}\\
&=&
\ds 
(\mathrm{sgn})
\det D'_1(x)
\det \pmat{\bs{w}&X_1}
\det\pmat{
\bs{y}&D_5(x)&0\\
\bs{z}&W'&X_4}.
\end{array}
\]
Then, we are able to continue a calculation on $\det \pmat{\bs{w}&X_1}$ as follows.
Since a signature of its determinant does not affect to a result we want to prove,
we omit to calculate signatures and write just $(\mathrm{sgn})$ instead. 
\[
\begin{array}{r@{\ }c@{\ }l}
\det \pmat{\bs{w}&X_1}
&=&
\ds
\det
\smat{
w_1&x_1\\
w_2&-x_1&x_2\\
\vdots&&\ddots&\ddots\\
\vdots&&&-x_{n-2}&x_{n-1}\\
w_n&&&&-x_{n-1}
}
=
x_1\cdots x_{n-1}
\det
\smat{
w_1&1\\
w_2&-1&1\\
\vdots&&\ddots&\ddots\\
\vdots&&&-1&1\\
w_n&&&&-1
}\\
&=&
\ds
\mathrm{(sgn)}
x_1\cdots x_{n-1}(w_1+\cdots+w_n).
\end{array}
\]
Recall a formula $\det\smat{A&B\\C&D}=\det A\det(D-BA^{-1}C)$ when $\det A\ne 0$ for block matrices.
Using this formula, we can proceed a calculation as follows.
\[
\begin{array}{r@{\ }c@{\ }l}
\det\pmat{
\bs{y}&D_5(x)&0\\
\bs{z}&W'&X_4}
&=&
(\mathrm{sgn})
\det\pmat{
D_5(x)&\bs{y}&0\\
W'&\bs{z}&X_4}\\
&=&
(\mathrm{sgn})
\det D_5(x)
\det\left(
\pmat{\bs{z}&X_4}-W'D_5(x)^{-1}\pmat{\bs{y}&0}
\right)\\
&=&
(\mathrm{sgn})\det D_5(x)
\det\left(
\bs{z}'|X_4
\right),
\end{array}
\]
where we set
\[
\bs{z}'
=
\pmat{
z_1+w_1\frac{y_0}{x_1}\\
z_2+(w_1+w_2)\frac{y_1}{x_2}\\
\vdots\\
z_{n-1}+(w_{n-2}+w_{n-1})\frac{y_{n-2}}{x_{n-1}}\\
z_n+(w_{n-1}+w_n)\frac{y_{n-1}}{x_n}+w_n\frac{y_n}{x_{n-1}}
}.
\]
Thus, we obtain
\[
\begin{array}{r@{\ }c@{\ }l}
\det(X_2|\bs{z}')
&=&
\det
\pmat{
z'_1&x_2&&&\\
z'_2&-x_0&x_3&&\\
&\ddots&\ddots&&\vdots\\
z'_{n-1}&&&-x_{n-3}&x_{n}\\
z'_n&&&&-x_{n-2}
}\\
&=&
\ds
\sum_{i=1}^n
(-1)^{i+1}z'_i
\prod_{j=i-1}^{n-2}(-x_j)
\prod_{k=2}^i x_k\\
&=&
\ds
x_2\cdots x_{n-2}\sum_{i=1}^n x_{i-1}x_iz'_i\\
&=&
\ds
x_2\cdots x_{n-2}
\left(
x_0x_1z_1+x_0w_1y_0+\sum_{i=2}^n\Bigl(
x_{i-1}x_iz_i+x_{i-1}w_{i-1}y_{i-1}+x_{i-1}w_iy_{i-1}
\Bigr)
+
x_nw_ny_n
\right)\\
&=&
\ds
x_2\cdots x_{n-2}
\sum_{i=1}^n
\Bigl(x_{i-1}x_iz_i+x_{i-1}y_{i-1}w_i+x_iy_iw_i\Bigr)
=
x_2\cdots x_{n-2}p(\bs{x})
\end{array}
\]
Summing up the above calculation,
we have obtained
\[
\det B(\bs{x})=(\mathrm{sgn})\,x_0x_1^3(x_2\cdots x_{n-1})^4x_n^2\,(w_1+\cdots+w_n)\,p(\bs{x}).
\]
This shows that a general rank of $A(\bs{x})$ is equal to $4n+2$,
which implies that the triplet $(\mathfrak{g}[p],\rho,V)$ is a prehomogeneous vector space.
By a structure of $\mathfrak{g}[p]$,
it is easily verified that polynomials
\[
x_0,x_1,\dots,x_n\quad\text{and}\quad
w_1+\cdots+w_n.
\]
which are irreducible factors of $\det B(\bs{x})$,
are relatively invariant under the action of $\mathfrak{g}[p]$.
\end{proof}



\newpage 

\begin{figure}
    \centering
    \includegraphics[scale=0.4]{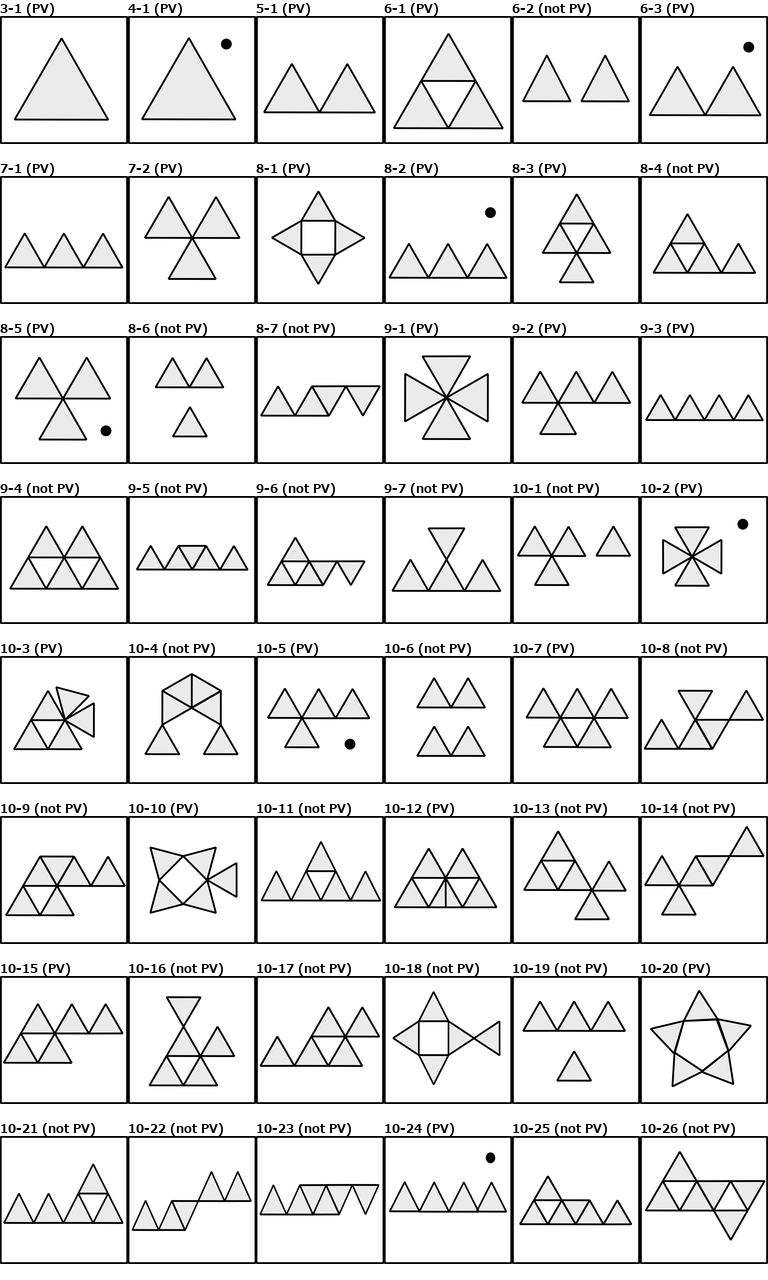}
    \caption{A table of triangle arrangements obtained by reduction of triangulation of $n$-polygon up to $n\le 10$.
    }
    \label{fig:trig}
\end{figure}
\end{document}